\documentclass[11pt,reqno]{amsart}
\usepackage{amssymb}
\usepackage{amsfonts}
\usepackage{mathrsfs}
\usepackage{amsmath}
\usepackage{graphicx}
\usepackage{hyperref}
\usepackage{float}
\usepackage{epstopdf}
\usepackage{color}
\usepackage{bm}
\usepackage{comment}
\usepackage{soul}

\allowdisplaybreaks
\newtheorem{theorem}{Theorem}[section]
\newtheorem{proposition}[theorem]{Proposition}
\newtheorem{lemma}[theorem]{Lemma}

\newtheorem{example}[theorem]{Example}
\newtheorem{definition}[theorem]{Definition}

\newtheorem{thm}{Theorem}

\def\diam{\mathrm{diam}}

\def\mcJ{\mathcal{J}}

\def\mcD{\mathcal{D}}
\def\mcE{\mathcal{E}}

\numberwithin{equation}{section}

\begin{document}
\title[Existence and Uniqueness of diffusions on the Julia sets of Misiurewicz-Sierpinski maps]{Existence and Uniqueness of diffusions on the Julia sets of Misiurewicz-Sierpinski maps}

\author{Shiping Cao}
\address{Department of Mathematics, Cornell University, Ithaca 14853, USA}
\email{sc2873@cornell.edu}

\author{Malte S. Ha\ss ler}
\address{Department of Mathematics, Cornell University, Ithaca 14853, USA}
\curraddr{} \email{mh2479@cornell.edu}
\thanks{}

\author{Hua Qiu}
\address{Department of Mathematics, Nanjing University, Nanjing, 210093, P. R. China.}
\curraddr{} \email{huaqiu@nju.edu.cn}

\author{Ely Sandine}
\address{Department of Mathematics, Cornell University, Ithaca 14853, USA}
\email{ebs95@cornell.edu}

\author{ Robert S. Strichartz }
\address{Department of Mathematics, Cornell University, Ithaca, 14853, U.S.A.}
\curraddr{} \email{str@math.cornell.edu}


\date{}


\begin{abstract}
We study the balanced resistance forms on the Julia sets of Misiurewicz-Sierpinski maps, which are self-similar resistance forms with equal weights. In particular, we use a theorem of Sabot to prove the existence and uniqueness of balanced forms on these Julia sets. We also provide an explorative study on  the resistance forms on the Julia sets of rational maps with periodic critical points.
\end{abstract}
\maketitle

\section{introduction}
The study of diffusion processes on fractals emerged as an independent research field in the late 80's. Initial interest in such processes came from mathematical physicists working in the theory of disordered media \cite{AO,HBA,RTP}. On self-similar sets, the pioneering works are the constructions of  Brownian motions on the Sierpinski gasket \cite{G,kus} originated by Kusuoka and Goldstein independently and later \cite{BP} by Barlow and Perkins, and on the Sierpinski carpet \cite{BB} by Barlow and Bass. See \cite{KZ} for an equivalent but different construction put forth by Kusuoka and Zhou at {about} the same time. The Sierpinski gasket is finitely ramified, meaning that the fractal can be disconnected  by the removal of finitely many points, and the construction was later extended to wider families {of fractals},  such as the nested fractals \cite{Lindstrom} by Lindstr{\o}m, and  the post-critically finite (p.c.f.) self-similar sets \cite{ki1,ki2} by Kigami. Due to the rough structure of the fractals, the diffusion processes move slower  on average than a standard Brownian motion on $\mathbb{R}^d$, see \cite{BB1,BP,FHK,HK,kum} for the associated transition density estimates. See books \cite{B,ki3,s3} for  systematic explorations { of the subject} and more bibliographies.

On p.c.f. self-similar sets, Kigami \cite{ki1,ki2} showed that Dirichlet forms can be constructed as limits of electrical networks on approximating graphs. The construction relies on determining a proper form on the initial graph, whose existence and uniqueness in general  is a difficult and fundamental problem in fractal analysis. On nested fractals, Lindstr{\o}m proved that there always exists a symmetric diffusion process \cite{Lindstrom}. The problem was also investigated  on nested fractals and p.c.f. self-similar sets by Metz \cite{M2} and Sabot \cite{Sabot} respectively. In particular, Sabot ingeniously proved the uniqueness of  a symmetric diffusion process of equal weights on nested fractals by introducing the notion of \textit{preserved relations}. See also \cite{M3} by Metz, and \cite{Pe} by Peirone for short proofs. For p.c.f. self-similar sets, the general problem of uniqueness is still open, see \cite{HMT,Sabot} for some sufficient conditions. The uniqueness theorem for non-p.c.f. self-similar sets is more difficult, see \cite{BBKT} for a positive answer for the generalized Sierpinski carpets.

In this paper, we utilize Sabot's techniques on nested fractals  to study the existence and uniqueness of diffusions on a new class of finitely ramified fractals, the Julia sets of Misiurewicz-Sierpinski maps,  as introduced in  \cite{DRS}. Let 
\[R_{\lambda,n,m}(z)=z^n+\frac{\lambda}{z^m},\quad n\geq 2,m\geq1,\lambda\in\mathbb{C}\]
be a rational map. A point $c\in \mathbb{C}$ is a \textit{critical point} if $R'_{\lambda,n,m}(c)=0$. We call $R_{\lambda,n,m}$ a \textit{Misiurewicz-Sierpinski map} (\textit{MS map} for short) if:
\vspace{0.15cm}

(MS1). each critical point of $R_{\lambda,n,m}$ is on the boundary of the immediate attracting basin of $\infty$;

(MS2). each critical point of $R_{\lambda,n,m}$ is strictly preperiodic.
\vspace{0.15cm}

\noindent The dynamics and topological properties of {the} Julia sets $K_{\lambda,n,m}$ ($K_\lambda$ for short) associated with {the} MS map{s} of $R_{\lambda,n,m}$ were studied in \cite{DL,DRS}. In particular, $K_\lambda$ is a generalized Sierpinski gasket, in the sense that it is a limit set obtained by a similar recursive process defined for the Sierpinski gasket, but applied instead to the closed unit disk as starting set and by removing polygons of $N$ sides.  Due to this, there is a natural p.c.f. self-similar structure on $K_\lambda$, with i.f.s. $\{F_i\}_{i=1}^{m+n}$, which will be described in Section 2. See Figure \ref{fig1} for some examples of such fractals, where the green blocks  denote critical points, and the orange blocks  denote orbits of critical points.

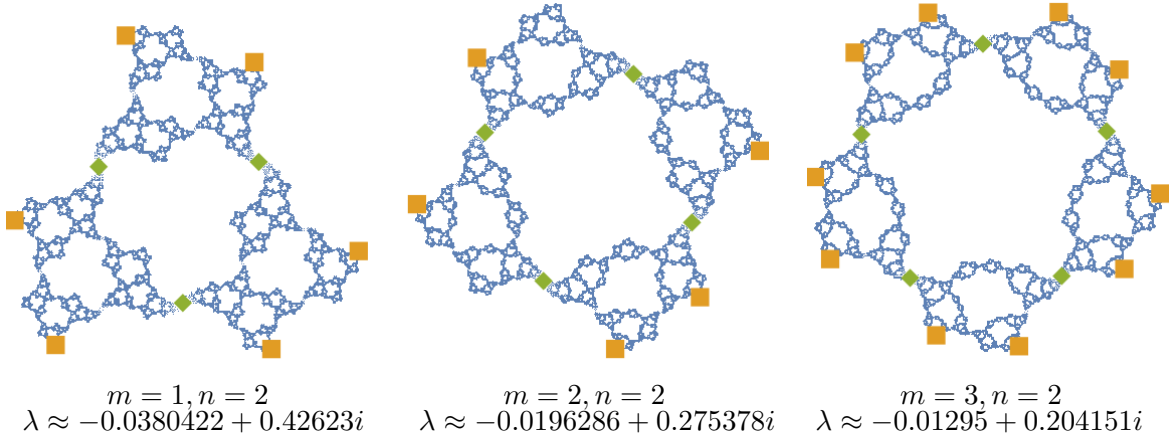
\begin{figure}[htp]
	\includegraphics[width=4.8cm]{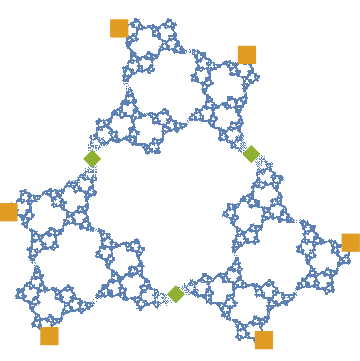}\hspace{0.4cm}
	\includegraphics[width=4.8cm]{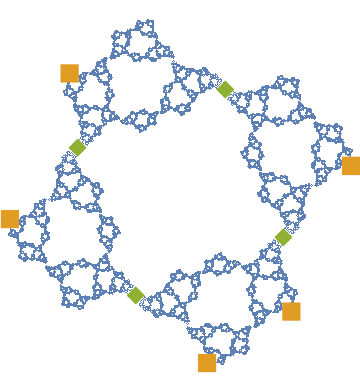}\hspace{0.4cm}
	\includegraphics[width=4.8cm]{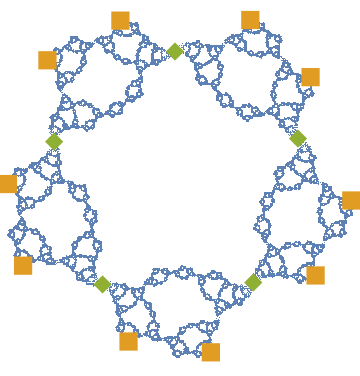}
	\begin{picture}(0,0)
	\put(-180,0){$m=1, n=2$}
	\put(-212,-10){$\lambda\approx -0.0380422+0.42623i$}
	\put(-30,0){$m=2, n=2$}
	\put(-62,-10){$\lambda\approx-0.0196286+0.275378i$}
	\put(120,0){$m=3, n=2$}
	\put(88,-10){$\lambda\approx -0.01295+0.204151 i$}
	\end{picture}
	\caption{Examples of MS Julia sets.}\label{fig1}
\end{figure}

The family of {Julia sets of} MS maps provides us a rich class of fractals. In particular, for fixed $m,n$ and MS parameters $\lambda,\tau$, {if $\tau\notin \{\lambda e^{\frac{2k\pi i}{n-1}},\bar{\lambda}e^{\frac{2k\pi i}{n-1}}:1\leq k\leq n-1\}$}, $K_\lambda$ and $K_\tau$ are not topological{ly} equivalent \cite{DRS}. The self-similar structure can be complicated depending on the choice of $\lambda$. Despite these difficulties, we will prove the existence and uniqueness of balanced resistance forms on such Julia sets.

\begin{thm}
Let $R_{\lambda,n,m}$ be a MS map, and $K_\lambda$ be the associated Julia set. There exists a unique resistance form $(\mathcal E_\lambda, \mathcal F_\lambda)$ on $K_\lambda$ such that 
\[\mcE_\lambda(f)=\eta\sum_{i=1}^{m+n}\mcE_\lambda(f\circ F_i), \quad \forall f\in\mathcal F_\lambda,\]
for some constant $\eta>1$. We call such a form a balanced form. 
\end{thm}

Though we are considering resistance forms with equal weights, our story has some essential differences with that  of nested fractals.
\vspace{0.15cm}

1. The uniqueness is a little stronger than that on nested fractals in that symmetry of the form is not required  by the problem. The same conclusion is not true in general on nested fractals, for example {the} Vicsek sets {admit} infinitely many different resistance forms with equal weights \cite{M1,Sabot}.   

2. {Compared to the MS Julia sets,} nested fractals have a larger symmetry group, {big enough} so that any pair of the boundary vertices are permuted by some element, and the existence  can be proven with a fixed point argument \cite{Lindstrom}. On the other hand, our proof of existence will use the full strength of Sabot's techniques. In particular, the proof depends crucially on the dynamics of $R_{\lambda,n,m}$ on $K_\lambda$. One main difficulty is to find all possible non-trivial preserved $\mathcal{G}$-relations.
\vspace{0.15cm}

We mention that there have been several previous works studying the resistance forms on Julia sets of polynomial maps \cite{ADS,FS,RT}, but the methods and goals are quite different from those in this paper. Additionally, at the end of this {work}, we  study other Julia sets associated with rational maps,  specifically those with fixed critical points. See Figure \ref{fig0} for an illustration.  The resistance forms on such Julia sets admit graph-directed structures. 

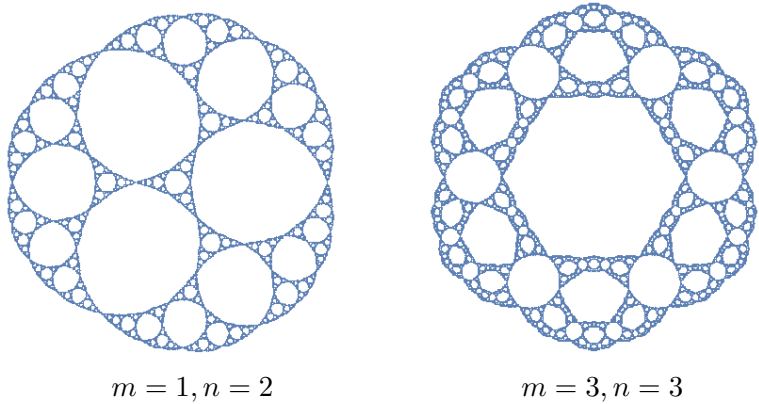
\begin{figure}[htp]
\includegraphics[width=4.5cm]{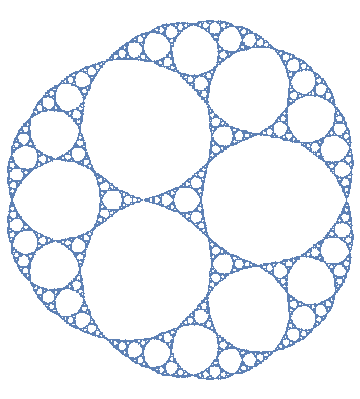}\hspace{1cm}
	\includegraphics[width=4.5cm]{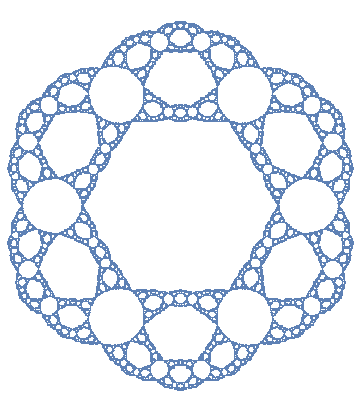}
	\begin{picture}(0,0)
	\put(-250,-10){$m=1, n=2$}
	\put(-95,-10){$m=3, n=3$}
	\end{picture}
	\caption{Julia sets of $R_{\lambda,n,m}$ with a fixed critical point.}\label{fig0}
\end{figure} 

We briefly introduce the structure of the paper. In Section 2, we will introduce the p.c.f. self-similar structures and some dynamic{ al} properties of the Julia sets of MS maps. Section 3 will be our main section, where we prove Theorem 1. This section will be divided into 4 parts. In the first part, we review the construction of resistance forms, and introduce Sabot's theorem. In the second part, we prove the existence of resistance forms. In the third part, we prove the uniqueness of resistance forms. Lastly, at the end of Section 3, we provide some examples. In Section 4, we present explorative results on the Julia sets of rational maps with a fixed critical point. Some rough discussions on the existence and non-existence of forms will be provided.

Throughout this paper, we will write $R_\lambda$ instead of $R_{\lambda,n,m}$ {if no confusion is likely.}

\section{A review of Misiurewicz-Sierpinski maps}\label{sec2}
In this section, we briefly review some simple properties of MS maps, and introduce self-similar structures on their associated Julia sets. Readers can find more details in \cite{DRS}.

Let $R_\lambda(z)=z^n+\frac\lambda {z^m}$ be a MS map and $K_\lambda$  the associated Julia set. Recall that a point $c {\in \mathbb C}$ is a \textit{critical point} if
\[R'_\lambda(c)=nz^{n-1}-m\frac{\lambda}{z^{m+1}}=0.\]
Let $C$ be the set of critical points, excluding the poles at $0$ and $\infty$. We have $\# C=m+n$, and $C$ admits rotational symmetry about $z=0$, i.e. $e^{\frac{2\pi i}{m+n}}C=C$. Moreover, we have the following proposition (see Theorem 3.3 in \cite{DRS} with $m=n=2$, which holds in general with  an almost identical proof).

\begin{proposition}
The critical set $C$ is the only set of $m+n$ points in the Julia set whose removal disconnects $K_\lambda$ into exactly $m+n$ components.
\end{proposition}

Denote the $m+n$ components of $K_\lambda\setminus C$ as $\{\mathring{K}_{\lambda,1},\cdots,\mathring{K}_{\lambda,m+n}\}$. For each $i$, {we let} $K_{\lambda,i}$ {be} the closure of $\mathring{K}_{\lambda,i}$ and call it a \textit{$1$-cell} of $K_\lambda$. The map $R_\lambda$ is then a homeomorphism from $K_{\lambda,i}$ to $K_\lambda$, and we denote $F_i$ the $i$-th branch of $R_\lambda^{-1}$ from $K_\lambda$ to $K_{\lambda,i}$. We thus have
\[C=\bigcup_{i\neq j} F_iK_\lambda\cap F_jK_\lambda,\quad K_\lambda=\bigcup_{i=1}^{m+n} F_iK_\lambda.\]

Next, according to Proposition 3.6 in \cite{DRS}, we have
\[\diam(F_{\omega_1}F_{\omega_2}\cdots F_{\omega_k}K_\lambda)\to 0, \text{ as }k\to\infty,\]
for any infinite word $\omega\in \{1,2,\cdots,m+n\}^{\mathbb{N}}$. This provides an addressing map $\Lambda:\{1,2,\cdots,m+n\}^{\mathbb{N}}\to K_\lambda$ defined by
\[\{\Lambda(\omega)\}=\bigcap_{k=1}^\infty F_{\omega_1}F_{\omega_2}\cdots F_{\omega_k}K_\lambda.\]
Define $V_0=\bigcup_{k=1}^\infty R_\lambda^{\circ k}(C)$ and $\mathcal{P}=\Lambda^{-1}(V_0)$, then we have
\[\#\mathcal{P}=\# V_0<\infty.\]
In fact, by (MS1), we can see that $\# V_0<\infty$ and $V_0\cap\bigcup_{k=0}^\infty R_\lambda^{-k}(C)=\emptyset$, which implies that $\#\mathcal{P}=\# V_0$. This shows that $K_\lambda$ admits a natural post-critically finite (p.c.f. for short) self-similar structure with the iterated function system (i.f.s. for short) $\{F_i\}_{i=1}^{m+n}$. See Figure \ref{fig1} for some examples of $K_\lambda$, with $C$ ({green blocks}) and $V_0$ ({orange blocks}) marked.

To better understand the self-similar structure, we use the boundary $\beta_\lambda$ of the immediate attracting basin $B_\lambda$ of $\infty$, and refer to the fact that $\beta_\lambda$ is a simple closed curve \cite{DRS}. The critical set $C$ disconnects $\beta_\lambda$ into $m+n$ components, and each is  contained in a component of $K_\lambda$, say $K_{\lambda,i}$. Moreover, by suitably ordering $K_{\lambda,i}=F_iK_\lambda$ and critical points $c_i\in C$, we have the property that
\begin{equation}\label{eqn21}
\{c_{i-1},c_i\}=F_iK_\lambda\cap C, \text{ for } 1\leq i\leq m+n,
\end{equation}
where we use cyclic notation $m+n=0$. The $1$-cells of $K_\lambda$ form a `ring' shape, see Figure \ref{fig22} for an illustration.

\begin{figure}[htp]
\includegraphics[width=5.3cm]{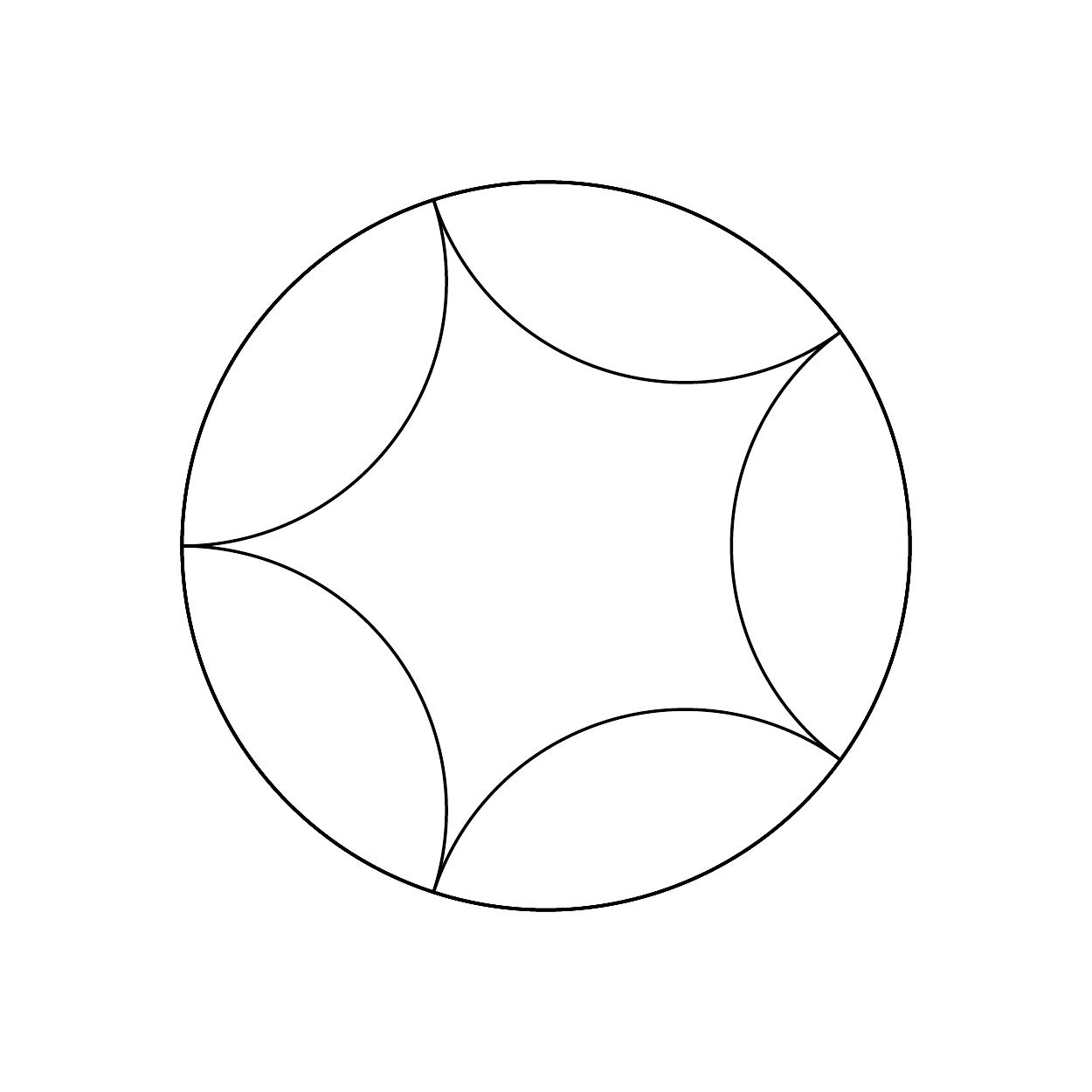}
\begin{picture}(0,0)
    \put(-20,25){$c_{m+n}$}
    \put(-35,70){$F_1K_\lambda$}
    \put(-80,18){$F_{m+n}K_\lambda$}
    \put(-16,117){$c_1$}
\end{picture}
\caption{An illustration of level-1 cells.}\label{fig22}
\end{figure}
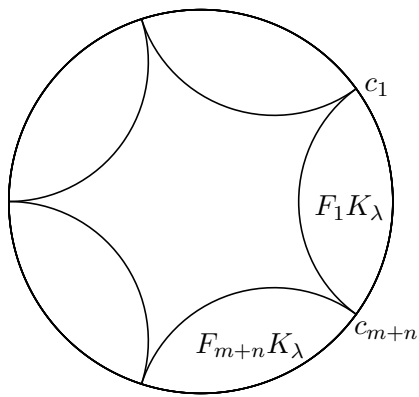

It is well-known that $R_\lambda:\beta_\lambda\to \beta_\lambda$ is conjugate to a simple dynamic on the unit { circle} $\mathbb{T}$. More precisely, we formally define the unit circle as
\[\mathbb{T}=\mathbb{R}/\mathbb{Z}=\{[r]=r+\mathbb{Z}: r\in \mathbb{R}\}.\]
For $n\geq 2$, we define $\Phi_n:\mathbb{T}\to \mathbb{T}$ by
\[\Phi_n([\theta])=[n\theta].\]
Then, there is a homeomorphism $\psi_{\lambda,n,m}:\beta_\lambda\to \mathbb{T}$ ($\psi_\lambda$ for short) such that
\[\psi_{\lambda}\circ R_\lambda=\Phi_n\circ \psi_\lambda.\]
 We denote by $\theta_\lambda$  the unique element of $\psi_\lambda(C)$ in $[0,\frac{1}{m+n})$. Since each $c\in C$ is strictly preperiodic, we have $\theta_\lambda\in \mathbb{Q}$. In addition,
\[\begin{cases}
\psi_\lambda(C)=\{[\theta_\lambda+\frac{l}{m+n}]:0\leq l\leq m+n-1\},\\
\psi_\lambda(V_0)=\{[n^k(\theta_\lambda+\frac{l}{m+n})]:k\geq 1, 0\leq l\leq m+n-1\}.
\end{cases}\]

\begin{example}\label{example22}
(a). The first {image} in Figure \ref{fig1} is the Julia set associated to $R_{\lambda,2,1}$ with $\lambda\approx -0.0380422+0.42623i$. For this simple example, we have $\theta_\lambda=\frac{1}{12}$. Thus,
\[\psi_\lambda(C)=\{[\frac{1}{12}],[\frac{5}{12}],[\frac{3}{4}]\}, \text{ and }\psi_\lambda(V_0)=\{[0],[\frac{1}{6}],[\frac{1}{3}],[\frac{1}{2}],[\frac{2}{3}],[\frac{5}{6}]\}.\]

(b). The second {image} in Figure \ref{fig1} is the Julia set associated to $R_{\lambda,2,2}$ with $\lambda\approx -0.0196286-0.275378i$, for which we have $\theta_\lambda=\frac{3}{16}$ and 
\[\psi_\lambda(C)=\{[\frac{3}{16}],[\frac{7}{16}],[\frac{11}{16}],[\frac{15}{16}]\}, \text{ and }\psi_\lambda(V_0)=\{[0],[\frac{3}{8}],[\frac{7}{8}],[\frac{3}{4}],[\frac{1}{2}]\}.\]
In particular, this example shows that $V_0$ may not be rotational symmetric.
\end{example}

It is often useful to  consider smaller cells. Let's focus on a $1$-cell $F_iK_\lambda$, which is bounded by $F_i\beta_\lambda$. Clearly $F_iK_\lambda$ is disconnected into exactly $m+n$ components, if we remove $F_iC$. Another observation is that $\psi_\lambda(\beta_\lambda\cap F_i K_\lambda)$ is a closed arc of length $\frac{1}{m+n}$, which contains exactly $n$ points in $\psi_\lambda(F_iC)$. That means we have
\[\# F_iC\cap \beta_\lambda=n,\]
and thus there are exactly $m$ points in $F_iC\setminus \beta_\lambda$. With this  in mind, we can sketch the level-$2$ cells, as illustrated in Figure \ref{fig23}.

\begin{figure}[htp]
	\includegraphics[width=5cm,angle=10]{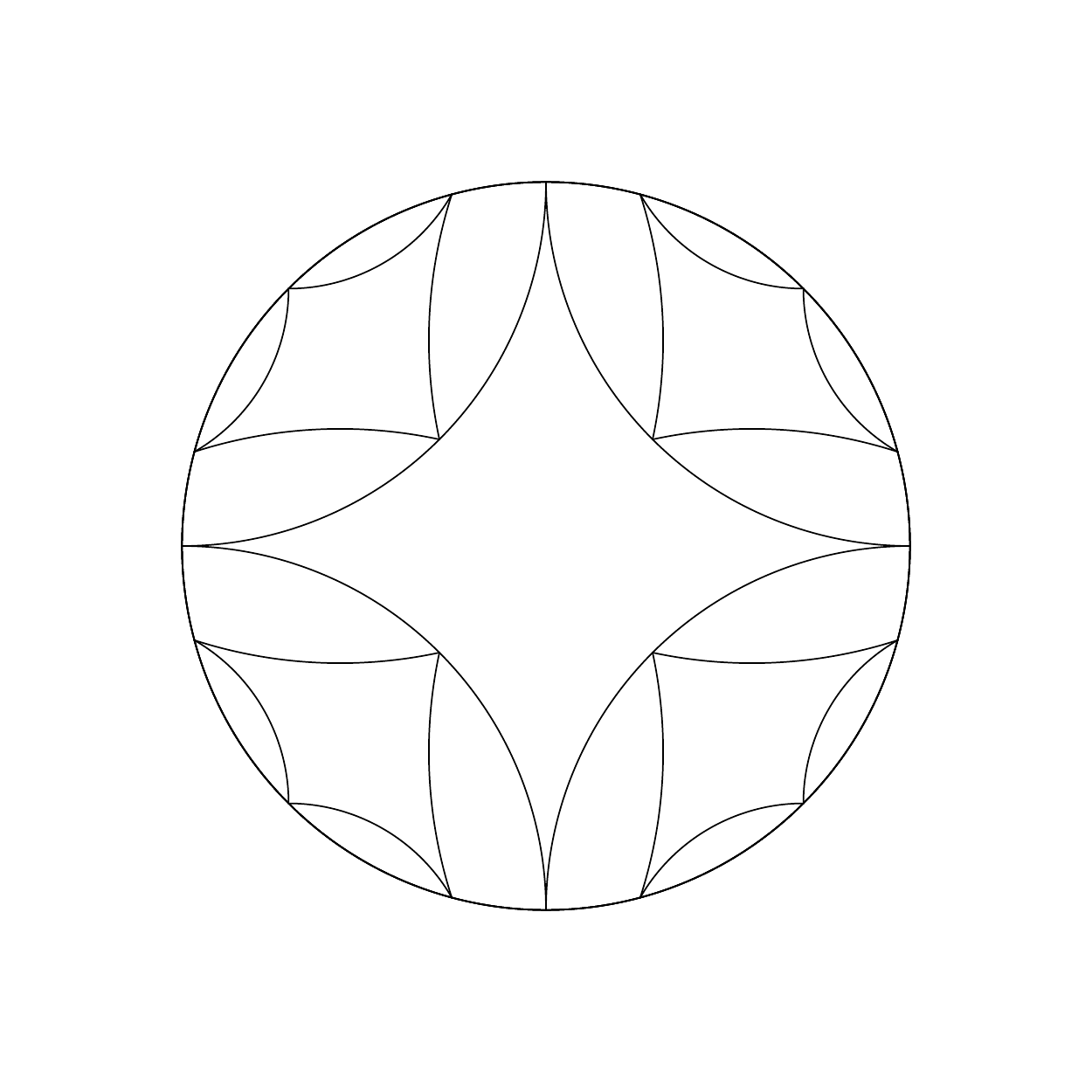}\qquad
	\includegraphics[width=5cm,angle=10]{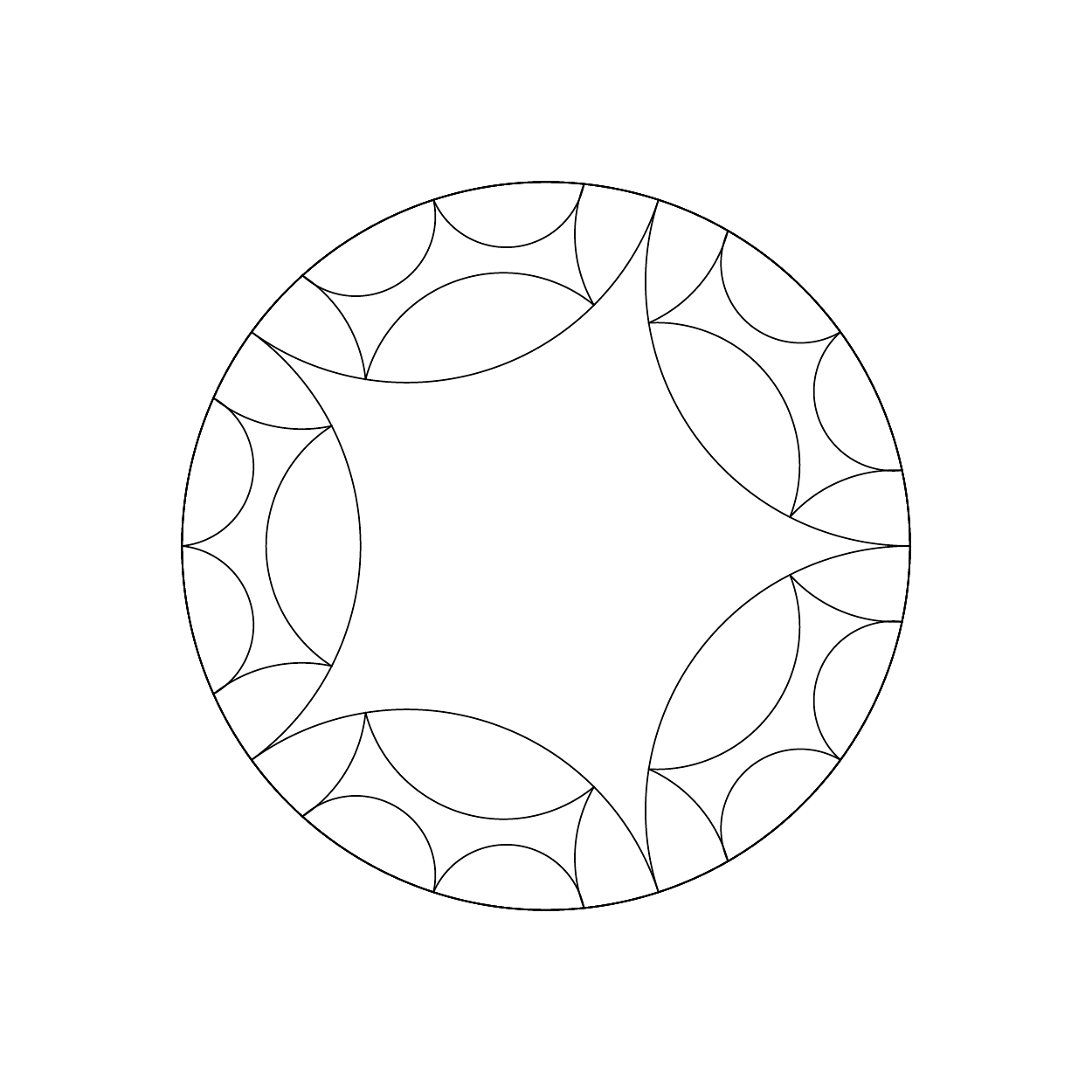}
	\caption{An illustration of the level-$2$ cells.}\label{fig23}
	\begin{picture}(0,0)
	\put(-123,33){$m=1, n=3$}
	\put(64,33){$m=2, n=3$}
    \end{picture}
\end{figure}
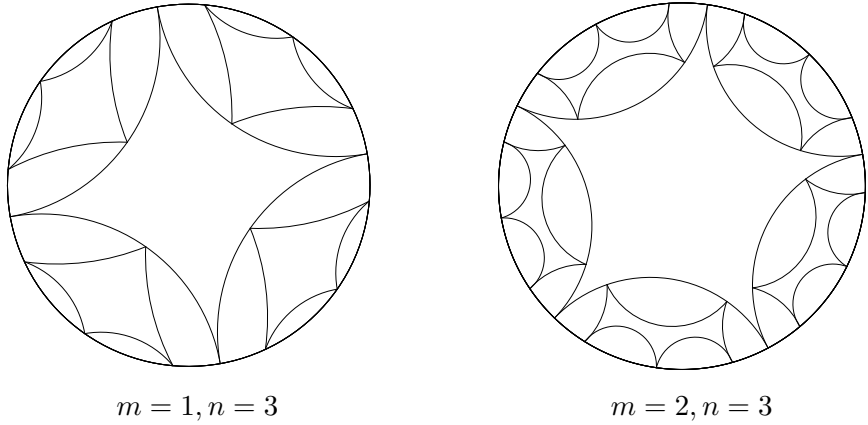

However, {to accurately sketch the} higher level cells, we need the exact value of $\theta_\lambda$. There could exist several cases, see Figure \ref{fig24}. In general, the exact {structure} can {depend in a complicated way on $\theta_\lambda$}. We end this section by enumerating some facts about Misiurewicz-Sierpinski parameters $\lambda$ and $\theta_\lambda$.

\begin{figure}[htp]\label{fig24}
	\includegraphics[width=5cm,angle=10]{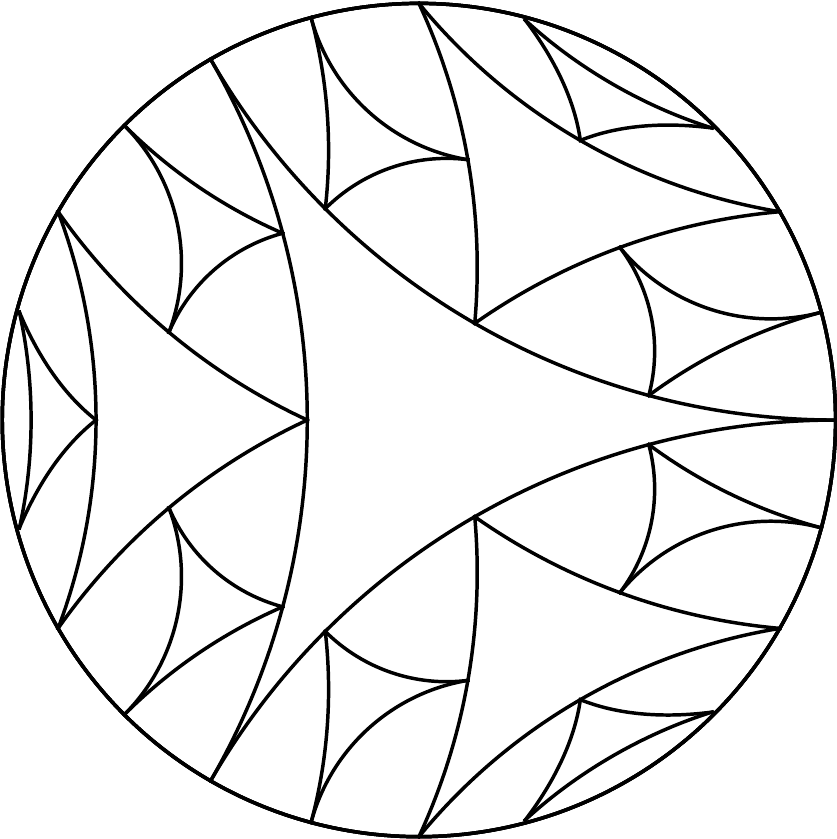}\qquad
	\includegraphics[width=5cm,angle=10]{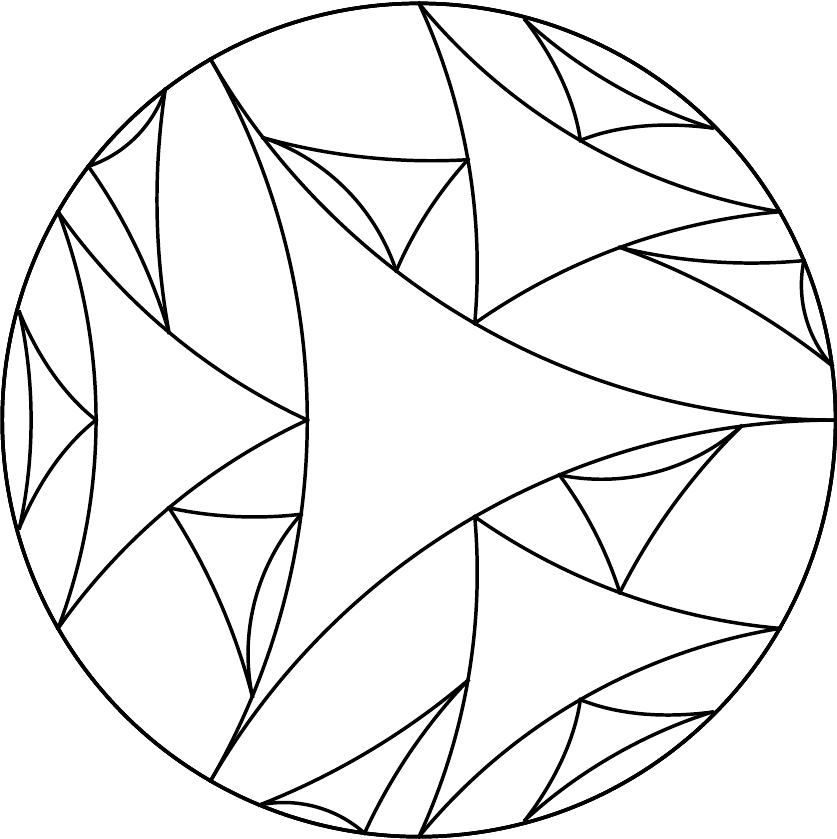}
	\caption{An illustration of the level-$3$ cells.}\label{fig24}
	\begin{picture}(0,0)
	\put(-141,33){$m=1, n=2,\theta_\lambda=\frac{1}{6}$}
	\put(46,33){$m=1, n=2,\theta_\lambda=\frac{1}{12}$}
	\end{picture}
\end{figure}
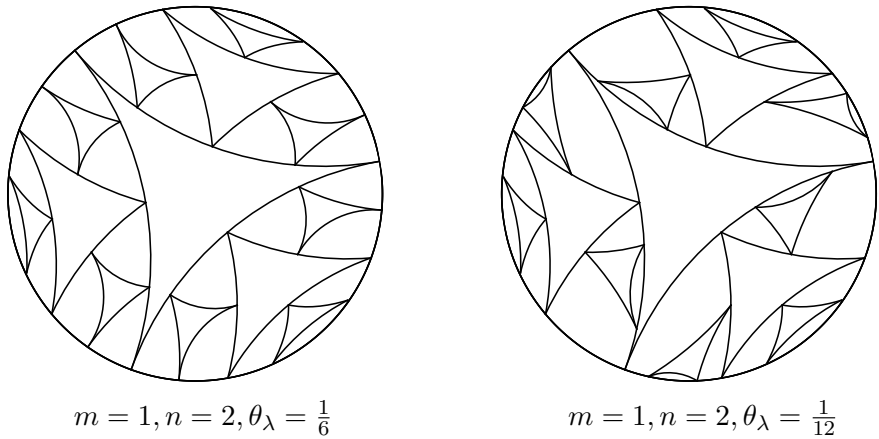
\vspace{0.15cm}

\noindent 1). (\cite{DRS,T}) The set of parameter values associated to MS maps is a dense subset {of} the boundary of the locus of connectedness of the family $R_{\lambda,n,m}$ (with fixed choice of $n,m$).

\noindent 2). (\cite{DRS}) For any two different MS maps $R_{\lambda,n,m}$ and $R_{\tau,n,m}$ with $\tau\notin \{\lambda e^{\frac{2k\pi i}{n-1}},\bar{\lambda}e^{\frac{2k\pi i}{n-1}}:1\leq k\leq n-1\}$, their respective Julia sets are not topologically equivalent.

\noindent 3). (\cite{QRWY}) For the $m=n$ case, there is a Misiurewicz-Sierpinski parameter $\lambda$ for any strictly  preperiodic $\theta_{\lambda}$ (under the dynamics of $\Phi_n$).

\section{Existence and uniqueness of balanced resistance forms}
In this section, we consider the existence and uniqueness of diffusions on  the Julia sets of Misiurewicz-Sierpinski maps. In \cite{ki3}, the concept of resistance forms {is} introduced, which in many cases describe local regular Dirichlet forms.

Let $X$ be a set, and $l(X)$ be the space of all real-valued functions on $X$. A pair $(\mathcal{E},\mathcal F)$ is called a \textit{non-degenerate resistance form} on $X$ if it satisfies the following  conditions:

\noindent (RF1). \textit{$\mathcal{F}$ is a linear subspace of $l(X)$ containing constants and $\mathcal{E}$ is a nonnegative symmetric quadratic form on $\mathcal F$; $\mathcal{E}(u):=\mathcal{E}(u,u)=0$ if and only if $u$ is constant on X.}

\noindent (RF2). \textit{Let $\sim$ be an equivalent relation on $\mathcal{F}$ defined by $u\sim v$ if and only if $u-v$ is constant on X. Then $(\mathcal{F}/\sim, \mathcal{E})$ is a Hilbert space.}

\noindent (RF3). \textit{For any finite subset $V\subset X$ and for any $v\in l(V)$, there exists $u\in \mathcal{F}$ such that $u|_V=v$.}

\noindent (RF4). \textit{For any $p,q\in X$, $r(p,q):=\sup\{\frac{|u(p)-u(q)|^2}{\mathcal{E}(u)}:u\in \mathcal{F},\mathcal{E}(u)>0\}$ is finite.}

\noindent (RF5). (Markov property) \textit{If $u\in \mathcal{F}$, then $\bar{u}={ \min\{\max\{u,0\}, 1\}}\in \mathcal{F}$ and $\mathcal{E}(\bar{u})\leq \mathcal{E}(u)$.}

Since each cell of the Julia set is a copy of itself under  compositions of homeomorphisms of the form $F_i$, {and the $1$-cells are all} the same size, a natural choice of resistance forms  {are} the \textit{balanced resistance forms}, {defined as follows.}

\begin{definition}
Let $R_{\lambda,n,m}$ be a MS map and {let} $K_\lambda$ be the corresponding Julia set. We say a resistance form $(\mathcal{E},\mathcal{F})$ on $K_\lambda$ is balanced if there exists a positive constant $\eta$ such that
\[\mathcal{E}(f)=\eta\sum_{i=1}^{m+n}\mathcal{E}(f\circ F_i),\quad \forall f\in\mathcal F.\]
\end{definition}

The main purpose of this paper is to prove the existence and uniqueness of a balanced resistance form on $K_\lambda$.

\begin{theorem}\label{thm32}
	Let $R_\lambda$ be a MS map, and $K_\lambda$ be the corresponding Julia set. Then there is a unique balanced resistance form $(\mathcal{E}_\lambda,\mathcal{F}_\lambda)$ on $K_\lambda$. In addition, the unique balanced resistance form has {rotational} symmetry,
	\[\mathcal{E}_\lambda(f(e^{\frac{2\pi i}{m+n}}\bullet))=\mathcal{E}_\lambda(f), \quad\forall f\in \mathcal F_\lambda.\]
\end{theorem}

In addition, by using well established results in \cite{ki3}, we can easily see {the following result}. 
\begin{theorem}\label{thm33}
Let $R_\lambda$ be a MS map, $K_\lambda$ be the corresponding Julia set, and $(\mathcal{E}_\lambda,\mathcal{F}_\lambda)$ be the corresponding balanced resistance form. Let $\mu$ be a Radon measure on $K_\lambda$. Then $(\mathcal{E}_\lambda,\mathcal{F}_\lambda)$ becomes a local regular Dirichlet form on $L^2(K_\lambda,\mu)$.
\end{theorem}

\subsection{The Theorem of Sabot}\label{sec31}
{We first demonstrate that the} problem of finding a balanced resistance form (or more generally, a self-similar resistance form) can be transferred to a nonlinear fixed point problem on a finitely dimensional space.

\vspace{0.2cm}
\noindent\textbf{Notation.}
	Let $X$ be a set equipped with a resistance form $(\mathcal{E},\mathcal{F})$, and $V\subset X$ be a subset of finitely many points. We define the \textit{restriction} of $(\mathcal E, \mathcal F)$ to $V$ by
	\[\mathcal{E}|_V(f)=\inf\{\mathcal{E}(\tilde{f}): \tilde{f}|_V=f,\tilde{f}\in \mathcal{F}\},\quad \forall f\in l(V).\]
	Note that $(\mathcal E|_V, l(V))$ is a resistance form on $V$ by the polarization identity. For $f\in l(V)$, we denote  the unique (\textit{harmonic}) \textit{extension} of $f$ with minimal energy by $H_{\mathcal{E},V}f$,  so that
	\[\mathcal{E}(H_{\mathcal{E},V}f)=\mathcal{E}|_V(f).\]
	Though we do not highlight $\mathcal{F}$ in the above notations,  the constructions depend on both $\mathcal{E}$ and $\mathcal{F}$.
\vspace{0.2cm}

Now, {we} assume that $(\mathcal{E},\mathcal{F})$ is a balanced resistance form on the Julia set $K_\lambda$. For each $f\in l(V_1)$, where $V_1=\bigcup_{i=1}^{m+n}F_i(V_0)$, we have
{\[\mathcal{E}|_{V_1}(f)=\mathcal{E}(H_{\mathcal{E},V_1}f)=\eta\sum_{i=1}^{m+n}\mathcal{E}\big((H_{\mathcal{E},V_1}f)\circ F_i\big)\geq \eta\sum_{i=1}^{m+n}\mathcal{E}\big(H_{\mathcal{E},V_0}(f\circ F_i)\big)=\eta\sum_{i=1}^{m+n}\mathcal{E}|_{V_0}(f\circ F_i).\]
The other direction inequality also holds,
\[\mathcal{E}|_{V_1}(f)=\mathcal{E}(H_{\mathcal{E},V_1}f)\leq \mathcal{E}(g)=\eta\sum_{i=1}^{m+n}\mathcal{E}\big(H_{\mathcal{E},V_0}(f\circ F_i)\big)=\eta\sum_{i=1}^{m+n}\mathcal{E}|_{V_0}(f\circ F_i).\]
where $g$ is the extension of $f$ from $l(V_1)$ to $\mathcal{F}$ such that $g\circ F_i=H_{\mathcal{E},V_0}(f\circ F_i)$. Thus we have
 $(H_{\mathcal{E},V_1}f)\circ F_i=H_{\mathcal{E},V_0}(f\circ F_i)$, and
\begin{equation}\label{eqn31}
\mathcal{E}|_{V_1}(f)=\eta\sum_{i=1}^{m+n}\mathcal{E}|_{V_0}(f\circ F_i).
\end{equation}
Notice that the equation $\big(\mathcal{E}|_{V_1}\big)|_{V_0}=\mathcal{E}|_{V_0}$ always holds, so {that} (\ref{eqn31}) becomes an identity of $\mathcal{E}|_{V_0}$.

On the other hand, if there exists a resistance form $\mathcal{D}$ on $V_0$ (we omit the domain $l(V_0)$ hereafter for simplicity) such that
\begin{equation}\label{eqn32}
T\mathcal{D}:=\mathcal{D}^{(1)}|_{V_0}=\eta^{-1}\mathcal{D},
\end{equation}
where $\mathcal{D}^{(1)}$ is the resistance form on $V_1$ defined by $\mathcal{D}^{(1)}(f)=\sum_{i=1}^{m+n}\mathcal{D}(f\circ F_i)$, and $\eta$ is a positive constant, then there exists a unique balanced resistance  form $(\mathcal{E},\mathcal{F})$ on $K_\lambda$ such that $\mathcal{D}=\mathcal{E}|_{V_0}$.

\vspace{0.15cm}
\noindent\textbf{Remark.}
	For the {purpose} of using rotation{al} symmetry {later}, we {sometimes} enlarge $V_0$ to $\tilde{V}_0:=\bigcup_{l=0}^{m+n-1} e^{\frac{2l\pi  i}{m+n}}V_0$ and consider $\mathcal{D}$ on $\tilde{V}_0$ instead. Clearly, $\tilde{V}_1:=\bigcup_{i=1}^{m+n}F_i\tilde{V}_0=R_\lambda^{-1}\tilde{V}_0$ is also rotationally symmetric. {We} note that $\tilde V_0=V_0$ when $m+n$ and $n$ are coprime.
	\vspace{0.15cm}

From equation (\ref{eqn32}), the problem becomes a fixed point problem of the map  $T$ on the projective space of resistance forms (Dirichlet forms) on $V_0$ (or $\tilde V_0$). {This problem} is of fundamental importance in the study of diffusions on finitely ramified fractals. A famous and pioneering work on the existence {of a solution} was written by Lindstr\o m for nested fractals \cite{Lindstrom}, which are a class of highly symmetric p.c.f. fractals. Later, the problem of uniqueness for nested fractals was solved by Sabot in his celebrated work \cite{Sabot}. Moreover, Sabot raised a general theorem on both the existence (also non-existence) and the uniqueness of $\mathcal D$ to the solution of (\ref{eqn32}).

We will utilize Sabot's theorem in our situation. Let's introduce some definitions from \cite{Sabot}.

Consider a general p.c.f. fractal $K$, associated with an i.f.s. $\{F_i\}_{i=1}^N$. Let $V_0=\Lambda(\mathcal{P})$ be the {set of} boundary vertices, and $V_1=\bigcup_{i=1}^N F_iV_0$. Assume that $\mathcal G$ is a finite group of homeomorphisms $K\to K$, such that $g(V_0)=V_0, \forall g\in \mathcal G$. In addition, we require that for any $g\in \mathcal G$ and $1\leq i\leq N$, there exists $g'\in \mathcal G$ and $1\leq i'\leq N$  such that
	\[g\circ F_i=F_{i'}\circ g'.\]

\begin{definition}\label{def35}
	Let $\mathcal{J}$ be an equivalence relation on $V_0$.
	
	(a). We define $\mathcal{J}^{(1)}$  to be the smallest equivalence relation on $V_1$ such that
	\[x \mathcal{J} y\Longrightarrow F_i(x) \mathcal{J}^{(1)}F_i(y), \quad 1\leq i\leq N.\]
	
	(b). We call $\mathcal{J}$ a preserved relation if for any $x,y \in V_0$,
	\[x \mathcal{J} y\Longleftrightarrow x \mathcal{J}^{(1)}y.\]
	In addition, if
	\[x \mathcal{J} y\Longrightarrow g(x) \mathcal{J} g(y),\quad \forall g\in \mathcal G,\]
	we call $\mathcal{J}$ a preserved $\mathcal G$-relation.
\end{definition}

\noindent\textbf{Remark.} We say that $\mathcal J$ is \textit{non-trivial} if $\mathcal J$ is neither the full relation  ($x\mathcal J y$, $\forall x,y\in V_0$), denoted by $\mcJ=1$, nor the null relation ($x\mcJ\mkern-10.5mu\backslash y$ if $x\neq y$), denoted by $\mcJ=0$.
\vspace{0.15cm}

A resistance form on a finite set $V$ can always be written as
\begin{equation}\label{eqn33}
\mathcal{D}(f)=\sum_{x\neq y}j_{x,y}(f(x)-f(y))^2,
\end{equation}
with nonnegative constants $j_{x,y}$. Noticing that by definition, a resistance form is non-degenerate, i.e. we always have $\mathcal{D}(f)=0$ if and only if $f$ is a constant function. There are also \textit{degenerate forms} of the form (\ref{eqn33}) whose kernel {is} larger than the space of constant functions. Clearly, $\mathcal D$ is non-degenerate if and only if the matrix $(j_{x,y})_{x,y\in V}$ is irreducible.

\begin{definition}
	Let $\mathcal J$ be a preserved relation on $V_0$, and $\mathcal{D}$ be a form of the form (\ref{eqn33}).
	
	(a1). We say $\mathcal{D}\in \mathcal{M}_\mcJ$ if
	\[\mathcal{D}(f)=0\Longleftrightarrow f\text{ is constant on each equivalence class of }\mathcal{J}.\]
	
	(a2). We define $T_\mcJ:\mathcal{M}_\mcJ\to \mathcal{M}_\mcJ$ as follows,
	\[T_\mcJ\mathcal{D}(f)=\inf\{\mathcal{D}^{(1)}(\tilde{f}):\tilde{f}=f \text{ on }V_0, \tilde{f}\in l(V_1)\}, \]
	where $\mathcal{D}^{(1)}({\tilde{f}})=\sum_{i=1}^{m+n}\mathcal{D}({\tilde{f}}\circ F_i)$.
	
	(b1). Let $\mathcal{M}_{V_0/\mcJ}$ be the space of resistance forms on $V_0/\mcJ$. We identify $l(V_0/\mcJ)$ with the subspace of $l(V_0)$ where each $f$ admits constant values on each equivalence class of $\mcJ$. Then, for each resistance form $\mathcal{D}$ on $V_0$, we can naturally define
	\[\mcD_{V_0/\mcJ}(f)=\mcD(f), \quad \forall f\in l(V_0/\mcJ).\]
	Conversely, any form in $\mathcal{M}_{V_0/\mcJ}$ can be constructed in this way.
	
	(b2). We define $T_{V_0/\mcJ}:\mathcal{M}_{V_0/\mcJ}\to \mathcal{M}_{V_0/\mcJ}$ as follows,
	\[T_{V_0/\mcJ}\mathcal{D}_{V_0/\mcJ}(f)=\inf\{\mathcal{D}^{(1)}(\tilde{f}):\tilde{f}=f \text{ on }V_0, \tilde{f}\in l(V_1/\mcJ^{(1)})\}, \]
	where $\mathcal{D}^{(1)}({\tilde{f}})=\sum_{i=1}^{m+n}\mathcal{D}({\tilde{f}}\circ F_i)$.
\end{definition}

\begin{definition}
Let $\mcJ$ be a preserved relation on $V_0$.

(a). We define
\[\underline{\rho}_\mcJ(\mcD)=\inf_{f\in l(V_0)\setminus l(V_0/\mcJ)}\frac{T_\mcJ\mcD(f)}{\mcD(f)},\quad \overline{\rho}_\mcJ(\mcD)=\sup_{f\in l(V_0)\setminus l(V_0/\mcJ)}\frac{T_\mcJ\mcD(f)}{\mcD(f)}, \text{ for }\mcD\in \mathcal{M}_\mcJ,\]
and
\[\underline{\rho}_{V_0/\mcJ}(\mcD)=\inf_{f\in l(V_0/\mcJ)}\frac{T_{V_0/\mcJ}\mcD(f)}{\mcD(f)},\quad \overline{\rho}_{V_0/\mcJ}(\mcD)=\sup_{f\in l(V_0/\mcJ)}\frac{T_{V_0/\mcJ}\mcD(f)}{\mcD(f)}, \text{ for }\mcD\in \mathcal{M}_{V_0/\mcJ}.\]

(b). We define
\[\begin{aligned}
\underline{\rho}_\mcJ=\sup_{\mcD\in \mathcal{M}_\mcJ}\underline{\rho}_\mcJ(\mcD),&\quad \overline{\rho}_\mcJ=\inf_{\mcD\in \mathcal{M}_\mcJ}\overline{\rho}_\mcJ(\mcD),\\
\underline{\rho}_{V_0/\mcJ}=\sup_{\mcD\in \mathcal{M}_{V_0/\mcJ}}\underline{\rho}_{V_0/\mcJ}(\mcD),&\quad \overline{\rho}_{V_0/\mcJ}=\inf_{\mcD\in \mathcal{M}_{V_0/\mcJ}}\overline{\rho}_{V_0/\mcJ}(\mcD).
\end{aligned}\]

(c). If $\mcJ$ is in addition a preserved $\mathcal G$-relation, we define  $\underline{\rho}^{\mathcal G}_\mcJ=\sup_{\mcD}\underline{\rho}_\mcJ(\mcD)$ where the supremum is taken over $\mathcal G$-symmetric forms in $\mathcal{M}_\mcJ$. In particular, when $\mathcal G$ is taken to be the trivial group, i.e. $\mathcal G=\{id\}$, then $\underline{\rho}^{\mathcal G}_{\mcJ}=\underline{\rho}_{\mcJ}$.

$\overline{\rho}^{\mathcal G}_\mcJ$, $\underline{\rho}^{\mathcal G}_{V_0/\mcJ}$ and $\overline{\rho}^{\mathcal G}_{V_0/\mcJ}$ are defined in a same way.
\end{definition}

We now quote the theorem of Sabot which will serve as our main {instrument} for proving {the} existence and uniqueness.

\begin{theorem}[\cite{Sabot}]\label{thm38}
(a). If there exist two non-trivial preserved $\mathcal G$-relations $\mcJ$ and $\mcJ'$ on $V_0$, such that $\underline\rho^{\mathcal G}_{V_0/\mcJ}<\underline\rho^{\mathcal G}_{\mcJ'}$, then (\ref{eqn32}) does not have a $\mathcal G$-symmetric solution.

(b). If for all non-trivial preserved $\mathcal G$-relation $\mcJ$, it holds that $\overline{\rho}^{\mathcal G}_\mcJ<\underline{\rho}^{\mathcal G}_{V_0/\mcJ}$, then (\ref{eqn32}) has at most one $\mathcal G$-symmetric solution (up to a multiplicative constant). If moreover, there do not exist two strictly ordered non-trivial $\mathcal G$-relation{s} (i.e. $\mcJ\subset \mcJ'$ and $\mcJ\neq \mcJ'$), then we have exactly one $\mathcal G$-symmetric solution to (\ref{eqn32}).
\end{theorem}

\noindent\textbf{Remark.} For the uniqueness part, the inequality in Theorem \ref{thm38} can be loosed to \[\overline{\rho}_{\mcJ,k}^{\mathcal G}<\underline{\rho}_{V_0/\mcJ,k}^{\mathcal G}.\]
Here $\overline{\rho}_{\mcJ,k}^{\mathcal G}=\inf\limits_{\mcD\in \mathcal{M}_\mcJ}\sup\limits_{f\in l(V_0)\setminus l(V_0/\mcJ)}\frac{T^k_\mcJ\mcD(f)}{\mcD(f)}$ and $\underline{\rho}_{V_0/\mcJ,k}^{\mathcal G}=\sup\limits_{\mcD\in \mathcal{M}_{V_0/\mcJ}}\inf\limits_{f\in l(V_0/\mcJ)}\frac{T^k_{V_0/\mcJ}\mcD(f)}{\mcD(f)}$, where the supremum and infimum are taken over $\mathcal G$-symmetric forms. Indeed, readers can revise Lemma 5.7 in Sabot's paper \cite{Sabot} with this new assumption. The rest of the proof of the uniqueness in Section 5.4 {of} \cite{Sabot} then follows in the same way.

\subsection{Proof of existence} We return to the study of {the} Julia sets $K_\lambda$ associated with {the} MS maps of the form $R_\lambda(z)=z^n+\frac{\lambda}{z^m}$ with $n\geq 2, m\geq 1$ and $\lambda\in \mathbb{C}$. In this subsection, we will prove the existence of a balanced resistance form on $K_\lambda$. By the discussion in Section \ref{sec31}, it is enough to study the equation (\ref{eqn32}) by applying Sabot's theorem. However, since the fractal  can depend in a complicated manner on $R_{\lambda}$, we must be careful in verifying the conditions in Theorem \ref{thm38}.

Throughout this subsection, we will let $\mathcal G$ be the canonical rotation group on $K_\lambda$, that is $\mathcal G=\{g: g(x)=e^{\frac{2l\pi  i}{m+n}}x, 0\leq l\leq m+n-1\}$. We first note several properties of preserved relations on $V_0$.

\begin{definition}\label{def39}
 Let $G=(V,E)$ and $G'=(V',E')$ be two finite graphs, and $f: V\to V'$.

 (a). We define $f(G)=(f(V),f(E))$ to be the graph with vertices $f(V)$ and \[{f(E)=\big\{\{f(x),f(y)\}:\{x,y\}\in E,f(x)\neq f(y)\big\}.}\]

 (b). We define $G\cup G'$ {to be the} graph with vertices $V\cup V'$ and edges $E\cup E'$.
\end{definition}

\begin{definition}\label{def310}
Let $\mathcal J$ be a preserved relation on $V_0$.

(a). Define $G_\mcJ=(V_0,E_\mcJ)$ be a graph with vertices $V_0$, and $\{x,y\}\in E_\mcJ$ if and only if $x\mcJ y$.

(b). For $k\geq 0$, we define $V_k=\bigcup_{|w|=k}F_wV_0$ and {the} graph $G_{\mcJ}^{(k)}(=(V_k,E_\mcJ^{(k)}))=\bigcup_{|w|=k}F_kG_\mcJ$. Here we use $w$ to represent a finite word and $|w|$ to denote its length, so we sum over $w\in \{1,2,\cdots,m+n\}^k$. The notation $F_w$ is short for $F_{w_1}\circ F_{w_2}\circ\cdots \circ F_{w_k}$.

(c). We define an equivalence relation $\mcJ^{(k)}$ on $V_k$ by
\[x\mcJ^{(k)}y\Longleftrightarrow {x \text{ and }y\text{ belong to the same connected component of $G^{(k)}_\mcJ$}}.\]
For a sequence $x=x_0,x_1,x_2,\cdots,x_N=y$ such that $\{x_{i-1},x_i\}\in E_\mcJ^{(k)}$, we call it a $G_\mcJ^{(k)}$-path connecting $x$ and $y$.

\noindent We can also consider {a} preserved relation $\mcJ$ on $\tilde{V}_0$, and define $\tilde{G}_\mcJ$, $\tilde{G}_{\mcJ}^{(k)}$ and $\mcJ^{(k)}$ in the corresponding way.
\end{definition}

The following lemma, which shows the relation between $\mcJ^{(k)},k\geq 0$, will play a fundamental role throughout this section. In particular, (b), (c) are special properties of the Julia sets, {and do} not hold for general p.c.f. fractals.

\begin{lemma}\label{lemma311}
Let $\mcJ$ be a preserved relation on $V_0$ (or $\tilde{V}_0$). Then the definitions of $\mcJ^{(1)}$ in Definition \ref{def35} (a) and Definition \ref{def310} (c) coincide. More generally, for any $0\leq k<l$, we have $\mcJ^{(l)}$ is the smallest equivalence relation such that
\[x \mathcal{J}^{(k)} y\Longrightarrow F_w(x) \mathcal{J}^{(l)}F_w(y), \quad\forall |w|=l-k\text{ and }x,y\in V_k  \text{ (or } \tilde V_k\text{)}.\]
In addition, we have:

(a). Let $0\leq k<l$, then for any $x,y\in V_k$ (or $\tilde{V}_k$),
\[x\mcJ^{(k)}y\Longleftrightarrow x\mcJ^{(l)}y.\]

(b). Let $k\geq 1$, then for any  $x,y\in V_k$ (or $\tilde{V}_k$),
\[x\mcJ^{(k)}y\Longrightarrow R_\lambda(x)\mcJ^{(k-1)}R_\lambda(y).\]

(c). Let $k\geq 0$, then for any  $x,y\in V_k$ (or $\tilde{V}_k$),
\[x\mcJ^{(k)}y\Longrightarrow R_\lambda(x)\mcJ^{(k)}R_\lambda(y).\]
\end{lemma}
\begin{proof} Assume $x\mcJ^{(k)}y$, then there exists a $G_\mcJ^{(k)}$-path $x=x_0,x_1,x_2,\cdots,x_N=y$ connecting $x$ and $y$. Clearly, for any $|w|=l-k$, we have that $F_wx_0,F_wx_1,\cdots,F_wx_N$ is a $G_\mcJ^{(l)}$-path. This shows
\[x \mathcal{J}^{(k)} y\Longrightarrow F_w(x) \mathcal{J}^{(l)}F_w(y), \quad\forall |w|=l-k.\]
Noticing that $G_\mcJ^{(l)}=\bigcup_{|w|=l-k}F_w(G_\mcJ^{(k)})$, we have $E_\mcJ^{(l)}\subset \big\{\{F_w(x),F_w(y)\}:x\mcJ^{(k)}y,|w|=l-k\big\}$. Since $\mcJ^{(l)}$ is generated by the edge set $E_\mcJ^{(l)}$, we claim that $\mcJ^{(l)}$ is {the} smallest relation such that the above implication holds.
	
(a). We view an equivalence relation $\mcJ^{(k)}$ as a subset of $V_k\times V_k$. Then, we have $\mcJ\subset \mcJ^{(1)}$ as $\mcJ$ is preserved. Noticing that $\mcJ^{(1)}$ is generated with $\mcJ$, and $\mcJ^{(2)}$ is generated with $\mcJ^{(1)}$ in a same manner, we have $\mcJ^{(1)}\subset \mcJ^{(2)}$. Continuing the argument, we get
\[\mcJ\subset \mcJ^{(1)}\subset \mcJ^{(2)}\subset \mcJ^{(3)}\subset\cdots.\]
This shows $x\mcJ^{(k)}y\Longrightarrow x\mcJ^{(l)}y$.

For the other direction, we assume $x\mcJ^{(l)}y$. Then there exists a $G_\mcJ^{(l)}$-path $x=x_0,x_1,x_2,\cdots,x_N=y$ connecting $x$ and $y$. We choose a subsequence $x=x_{i_0}, x_{i_1},\cdots,x_{i_M}=y$ such that $0=i_0<i_1<i_2<\cdots<i_M=N$, and
\[\{x_{i_1},x_{i_2},\cdots,x_{i_{M-1}}\}=V_{l-1}\cap \{x_1,x_2,\cdots,x_{N-1}\}.\]
Now, we look at $x_{i_0}$ and $x_{i_1}$. Clearly,  $x_{i_0},x_{i_0+1},x_{i_0+2},\cdots,x_{i_1}$ is a $G_\mcJ^{(l)}$-path, which is contained in a same $(l-1)$-cell $F_wK_\lambda$. So $F_{w}^{-1}x_{i_0},F_{w}^{-1}x_{i_0+1},\cdots,F_{w}^{-1}x_{i_1}$ is a $G_\mcJ^{(1)}$-path, thus we have $F_{w}^{-1}x_{i_0}\mcJ^{(1)}F_{w}^{-1}x_{i_1}$, and so $F_{w}^{-1}x_{i_0}\mcJ F_{w}^{-1}x_{i_1}$. This implies $\{x_{i_0},x_{i_1}\}\in E_\mcJ^{(l-1)}$. By the same argument, we may show that $x=x_{i_0},x_{i_1},x_{i_2},\cdots,x_{i_M}=y$ is a $G_\mcJ^{(l-1)}$-path, so $x\mcJ^{(l-1)}y$. By repeating the above arguments, we have
\[x\mcJ^{(l)}y\Longrightarrow x\mcJ^{(l-1)}y\Longrightarrow \cdots \Longrightarrow x\mcJ^{(k)}y.\]

(b). Let $x\mcJ^{(k)}y$, then there is a $G^{(k)}_\mcJ$-path $x,x_1,x_2,\cdots,x_{N}=y$. Then we have $R_\lambda(x),$ $R_\lambda(x_1),\cdots,R_\lambda(y)$ is a $G_{\mcJ}^{(k-1)}$-path, noticing that $G_\mcJ^{(k)}=\bigcup_{i=1}^{m+n}F_iG_\mcJ^{(k-1)}$.

(c) {This assertion} is an easy {synthesis} of (a) and (b). We have for $x,y\in V_k$,
\[x\mcJ^{(k)}y\Longleftrightarrow x\mcJ^{(k+1)}y\Longrightarrow R_\lambda(x)\mcJ^{(k)}R_\lambda(y).\]

Finally, we point out that the proof for the $\tilde{V}_0$ setting is the same.
\end{proof}

As an important corollary to the Lemma \ref{lemma311}, we have the following lemma concerning the critical set $C$.

\begin{lemma}\label{lemma312}
	Let $\mcJ$ be a preserved relation on $V_0$ (or $\tilde{V}_0$). If $C$ is a subset of a class of $\mcJ^{(1)}$, then $\mcJ=1$. In particular, if $\mcJ$ is a non-trivial preserved $\mathcal G$-relation on $\tilde{V}_0$, we have $x{\mcJ\mkern-10.5mu\backslash}^{(1)} y$ for any distinct $x,y\in C$.
\end{lemma}
\begin{proof} First, assume $x\mcJ^{(1)}y$ for any $x,y\in C$. Then, by Lemma \ref{lemma311} (c), we have $x\mcJ^{(1)} y$ for any distinct $x,y\in R_\lambda(C)$. Note that due to the dynamics of $R_\lambda$, $x$ and $y$ could not belong to a same $1$-cell of $K_\lambda$. Let $x,x_1,x_2,\cdots,y$ be a $G^{(1)}_\mcJ$-path connecting $x$ and $y$, then the path exits the 1-cell $F_iK$ containing $x$ at some point $x_l\in C\cap F_iV_0$, and we have $x\mcJ^{(1)} x_l$. In particular, this argument implies that
\[x\mcJ^{(1)} y,\quad\forall x,y\in R_\lambda(C)\cup C.\]
In addition, by using Lemma \ref{lemma311} (c) again, for any $k\geq 0$, we still have
\[x\mcJ^{(1)} y,\quad\forall x,y\in R_\lambda^k\big(R_\lambda(C)\cup C\big).\]
Taking the union, we can see that
\[x\mcJ^{(1)} y,\quad\forall x,y\in \bigcup_{k=0}^\infty R^k_\lambda(C).\]
This implies
\[x\mcJ y,\quad\forall x,y\in V_0=\bigcup_{k=1}^\infty R^k_\lambda(C),\]
so $\mcJ=1$.

Next, we assume that $\mcJ$ is a non-trivial preserved $\mathcal G$-relation on $\tilde{V}_0$, and prove that $x{\mcJ\mkern-10.5mu\backslash}^{(1)} y$ for any distinct $x, y\in C$. For the sake of contradiction we assume there exists $x\neq y$ in $C$ such that $x\mcJ^{(1)} y$, and consider two cases.

\textit{Case 1:} $y=e^{\frac{2\pi i}{m+n}}x$. In this case, by rotation symmetry ($\mathcal G$-symmetry), we have
\[x\mcJ^{(1)} e^{\frac{2\pi i}{m+n}}x\mcJ^{(1)} e^{\frac{4\pi i}{m+n}}x\mcJ^{(1)}\cdots,\]
so $C$ is in a same class of $\mcJ^{(1)}$. This implies $\mcJ=1$, and gives a contradiction.

\textit{Case 2:} $y=e^{\frac{2k\pi i}{m+n}}x$ for some $2\leq k\leq m+n-2$. In this case, removing the vertices $\{e^{\pm \frac{2\pi i}{m+n}}x, e^{\pm \frac{2\pi i}{m+n}}y\}\subset C$ will disconnect $G^{(1)}_\mcJ$, so that $x,y$ belong to different components. This implies that we can find $y'=e^{\frac{2\pi i}{m+n}}x'$ in $C$ such that $x'\mcJ^{(1)} y'$, which {reduces the problem to Case 1.}
\end{proof}

Lemma \ref{lemma312} shows that for a non-trivial preserved $\mathcal G$-relation $\mcJ$, {the extension} $\mcJ^{(1)}$ is quite loose. In particular, there are few choices of paths for $x,y$ in different $1$-cells.

\begin{lemma}\label{lemma313}
Let $\mcJ$ be a non-trivial preserved $\mathcal{G}$-relation on ${\tilde{V}_0}$, and assume $x,y\in \tilde{V}_1\setminus C$ with $x\mcJ^{(1)} y$. Then there exists $1\leq i\leq m+n$ such that $\{x,y\}\subset F_i\tilde{V}_0\cup F_{i+1}\tilde{V}_0$ (cyclic notation $m+n+1=1$). In addition,  if $x,y$ belong to different $1$-cells, say $x\in F_i\tilde{V}_0$, $y\in F_{i+1}\tilde{V}_0$. Then we have $x\mcJ^{(1)}c_i\mcJ^{(1)} y$ (Recall (\ref{eqn21})). 
\end{lemma}

\begin{proof}
Let $x=x_0,x_1,\cdots,x_N=y$ be a $G^{(1)}_\mcJ$-path connecting $x$ and $y$. Assume $x\in F_i\tilde{V}_0$ and $y\in F_j\tilde{V}_0$. If $i\neq j$, then the path should leave $F_i\tilde{V}_0$ at $c_{i-1}$ or $c_{i}$, and enter $F_j\tilde{V}_0$ at $c_{j-1}$ or $c_j$. However, according to Lemma \ref{lemma312}, there is at most one critical point contained in the path (not counting multiplicity). This is only possible when $|j-i|=1$. 

If $x\in F_i\tilde{V}_0$ and $y\in F_{i+1}\tilde{V}_0$, then {any} path {from $x$ to $y$ includes} at least one $c\in C$, and the only possible choice is $c_i$ as discussed above.
\end{proof}

In the rest of this section, we will prove the existence of a $\mathcal G$-symmetric solution to (\ref{eqn32}). We consider the cases $m\geq 2$ and $m=1$ separately since they exhibit quite different properties.\vspace{0.2cm}

\noindent\textbf{The $m\geq 2$ case.} 
\vspace{0.15cm}

Throughout this part, we assume that $R_\lambda$ takes the form $R_\lambda(z)=z^n+\frac{\lambda}{z^m}$ with $n\geq 2,m\geq2$. In this case, we will show that there are only trivial preserved $\mathcal{G}$-relations. To use the dynamics more efficiently, we  introduce a natural distance on $\beta_\lambda$.

\begin{definition}\label{def314}
Let $x,y\in \beta_\lambda$, we define 
	\[d_{\beta_\lambda}(x,y)=d_\mathbb{T}(\psi_\lambda(x),\psi_\lambda(y)),\]
where $d_\mathbb{T}$ is the standard distance on the unit  circle $\mathbb{T}$ (for $a,b\in [0,1)$, we have $d_\mathbb{T}([a],[b])=\min\{|a-b|,1-|a-b|\}$).
\end{definition}

By the conjugacy of the dynamics of $R_\lambda$ and $\Phi_n$ on $\beta_\lambda$, it is easy to see the following result.
\begin{lemma}\label{lemma315}
Let $x,y\in \beta_\lambda$, we have $d_{\beta_\lambda}(R_\lambda(x),R_\lambda(y))=\min\{nd_{\beta_\lambda}(x,y),1-nd_{\beta_\lambda}(x,y)\}$ if $d_{\beta_\lambda}(x,y)<\frac{1}{n}$.
\end{lemma}

The distance allows us {to conveniently show when two vertices are in distinct $\mcJ^{(1)}$ classes. }

\begin{lemma}\label{lemma316}
Let  $\mathcal{J}$ be a non-trivial preserved $\mathcal{G}$-relation on $\tilde{V}_0$ and $x,y\in \tilde{V}_0$. We have 
\[d_{\beta_\lambda}(x,y)\geq \frac{2}{n(m+n)}\Longrightarrow x{\mcJ\mkern-10.5mu\backslash} y.\] 
\end{lemma}
\begin{proof}
By Lemma \ref{lemma311} (a) and Lemma \ref{lemma313}, without loss of generality, we  may assume that $x\in F_i\tilde V_0$ and $y\in F_i\tilde V_0\cup F_{i+1}\tilde V_0$. We consider two cases.\vspace{0.15cm}

\noindent\textit{Case 1: $y\in F_i\tilde{V}_0$.} In this case, $x$ and $y$ belong to the same $1$-cell, { so} $d_{\beta_\lambda}(x,y)<\frac{1}{m+n}$. Consequently, by Lemma \ref{lemma315}, we have 
\[d_{\beta_\lambda}(R_\lambda(x),R_\lambda(y))\geq \min\{n\cdot\frac{2}{n(m+n)},1-n\cdot\frac{1}{m+n}\}\geq \frac{2}{m+n}.\]
This means {that} $R_\lambda(x)$ and $R_\lambda(y)$ do not belong to neighbouring $1$-cells, so by Lemma \ref{lemma313}, $ R_\lambda(x){\mcJ\mkern-10.5mu\backslash}^{(1)} R_\lambda(y)$. Finally, we apply Lemma \ref{lemma311} to see $x{\mcJ\mkern-10.5mu\backslash}y$.\vspace{0.15cm}

\noindent\textit{Case 2: $y\in F_{i+1}\tilde{V}_0$.} We prove by contradiction. Assume that $x\mathcal{J}y$, then by  Lemma \ref{lemma313}, we have 
\[x\mathcal{J}^{(1)} c_i \mathcal{J}^{(1)}y.\]
See {Figure \ref{fig4}} for an illustration. We label the two vertices surrounding $c_i$ by $\tilde{c}_1,\tilde{c}_2\in R_\lambda^{-1}(C)\cap \beta_\lambda$, ordered so that $\tilde{c}_1\in F_i\tilde{V}_0$ and $\tilde{c}_2\in F_{i+1}\tilde{V}_0$. 

\begin{figure}[htp]
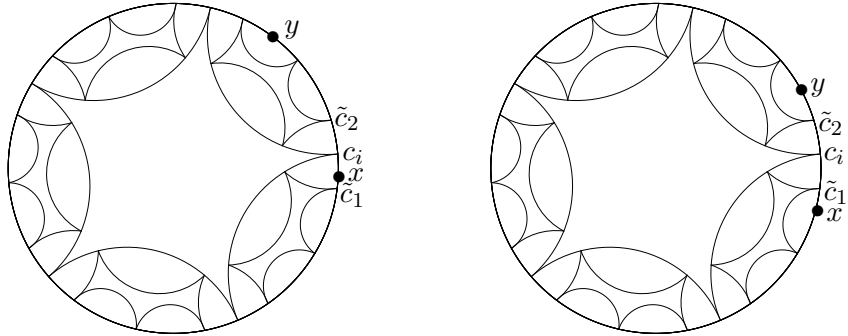

	\includegraphics[width=4.5cm,angle=5]{diagram3.pdf}\qquad\qquad
	\includegraphics[width=4.5cm,angle=5]{diagram3.pdf}
	\begin{picture}(0,0)
	\put(-196,63){$\bullet$}
	\put(-221,116){$\bullet$}
	\put(-190,64){$x$}
	\put(-214,120){$y$}
	\put(-192,72){$c_i$}
	\put(-193,56){$\tilde{c}_1$}
    \put(-195,85){$\tilde{c}_2$}	
	
	\put(-15,50){$\bullet$}
	\put(-21,96){$\bullet$}
	\put(-9,49){$x$}
	\put(-15,98){$y$}
	\put(-10,72){$c_i$}
	\put(-10,57){$\tilde{c}_1$}
	\put(-12,83){$\tilde{c}_2$}	
	\end{picture}

	\caption{An illustration of the vertices $x,\tilde{c}_1,c_i,\tilde{c}_2,y$.}\label{fig4}
\end{figure}

\textit{Case 2.1: $c_i,x$ belong to {the} same $2$-cell.} In this case, we claim that $c_i$ and $y$ do not belong to neighbouring $2$-cells. {If they do, then} 
$d_{\beta_\lambda}(x,c_i)<d_{\beta_\lambda}(c_i,\tilde{c}_1)$ and $d_{\beta_\lambda}(y,c_i)<d_{\beta_\lambda}(c_i,\tilde{c}_2)+\frac{1}{n(m+n)}$, which implies that $d_{\beta_\lambda}(x,y)\leq d_{\beta_\lambda}(x,c_i)+d_{\beta_\lambda}(y,c_i)< \frac{2}{n(m+n)}$, noticing that $d_{\beta_\lambda}(c_i,\tilde{c}_1)+d_{\beta_\lambda}(c_i,\tilde{c}_2)=\frac{1}{n(m+n)}$, {a contradiction}. It is then easy to see that $R_\lambda(c_i)$ and $R_\lambda(y)$ do not belong to neighbouring $1$-cells, and thus by Lemma \ref{lemma313}, $R_\lambda(c_i){\mcJ\mkern-10.5mu\backslash}^{(1)}R_\lambda(y)$. Then using Lemma \ref{lemma311} (c), we have $c_i{\mcJ\mkern-10.5mu\backslash}^{(1)}y$, {violating the initial assumption of Case 2}.

\textit{Case 2.2: $c_i,x$ belong to distinct $2$-cells.} We can additionally assume that $c_i,y$ also belong to distinct $2$-cells, otherwise {we are essentially back to} Case 2.1. Clearly by Lemma \ref{lemma311} (c), we have
\[x\mathcal{J}^{(1)}c_i\Longrightarrow R_\lambda(x)\mathcal{J}^{(1)}R_\lambda(c_i).\]
Then by Lemma \ref{lemma313} and Lemma \ref{lemma315}, $R_\lambda(x)$ and $R_\lambda(c_i)$ must belong to two neighbouring $1$-cells separately which intersection at $R_\lambda(\tilde c_1)$. This gives that
$R_\lambda(\tilde{c}_1)\mathcal{J}^{(1)}R_\lambda(c_i).$
By the symmetric argument we have $R_\lambda(\tilde{c}_2)\mathcal{J}^{(1)}R_\lambda(c_i)$. This implies $R_\lambda(\tilde{c}_1)\mcJ^{(1)}R_\lambda(\tilde{c}_2)$, contradicting Lemma \ref{lemma312}.
\end{proof}

\noindent\textbf{Remark.} The above proof indirectly uses $\mcJ^{(2)}$. More specifically, the proof of $x\mcJ^{(1)}y\Longrightarrow R_\lambda(x)\mcJ^{(1)} R_\lambda(y)$ in Lemma \ref{lemma311} (c) essentially involves {going to the second level.}\vspace{0.15cm}

By applying Lemma \ref{lemma315} and \ref{lemma316}, we can finally prove the non-existence of non-trivial preserved $\mathcal{G}$-relation when $m\geq 2$.

\begin{proposition}\label{prop317}
Let $R_\lambda(z)=z^n+\frac{\lambda}{z^m}$ with $m,n\geq 2$ be an MS map. There does not exist a non-trivial preserved $\mathcal{G}$-preserved relation on $\tilde{V}_0$. In particular, there exists exactly one $\mathcal G$-symmetric solution to (\ref{eqn32}).
\end{proposition}
\begin{proof}
    We {prove by} contradiction. Let $\mcJ$ be a non-trivial preserved $\mathcal{G}$-relation. Take $x\neq y$ from $\tilde{V}_0$. We claim that there exists $k\geq0$ such that 
	\begin{equation}\label{eqn34}
	d_{\beta_\lambda}(R_\lambda^k(x),R_\lambda^k(y))\geq \frac{2}{n(m+n)}.
	\end{equation}
	In fact, if not, then for all $k\geq 0$, we have $d_{\beta_\lambda}(R_\lambda^k(x),R_\lambda^k(y))<\frac{2}{n(m+n)}$,
	which implies 
	\[1-nd_{\beta_\lambda}(R_\lambda^{k-1}(x),R_\lambda^{k-1}(y))>1-n\frac{2}{n(m+n)}\geq  {\frac{2}{n(m+n)}}>d_{\beta_\lambda}(R_\lambda^{{k}}(x),R_\lambda^{k}(y)),\quad \forall k\geq 1.\] 
	Thus, by Lemma \ref{lemma315}, 
	\[d_{\beta_\lambda}(R_\lambda^k(x),R_\lambda^k(y))=nd_{\beta_\lambda}(R_\lambda^{k-1}(x),R_\lambda^{k-1}(y))=\cdots=n^kd_{\beta_\lambda}(x,y).\]
	Letting $k\to\infty$, we get $d_{\beta_\lambda}(x,y)=0$, {a} contradiction. 
	
	Now, to prove existence for the $m\geq 2$ case, choose $k\geq 0$ such that (\ref{eqn34}) holds. Then, by Lemma \ref{lemma316}, we have $R_\lambda^k(x){\mcJ\mkern-10.5mu\backslash}R_\lambda^k(y)$. Thus, applying Lemma \ref{lemma311} (c), we have $x{\mcJ\mkern-10.5mu\backslash}y$. Noticing that $x,y$ are arbitrarily chosen, we have $\mcJ=0$. A contradiction. 
	
	Finally, by applying Sabot's theorem, Theorem \ref{thm38} (b), there exists exactly one $\mathcal{G}$-symmetric solution to (\ref{eqn32}).
	\end{proof}

\noindent\textbf{The $m=1$ case.}
\vspace{0.15cm}

 In this case, there may exist a non-trivial preserved $\mathcal{G}$-relation on $\tilde V_0$ ($=V_0$). 

\begin{example}\label{example318}
	We consider the first Julia set presented in Example \ref{example22}. We define an equivalence relation $\mcJ$ by taking  each pair of `opposite' vertices to be a unique equivalence class. More precisely, there are three class $I_1,I_2,I_3$ in $\mcJ$, with  
	\[I_1=\psi_\lambda^{-1}\{[0],[\frac{1}{2}]\},\quad I_2=\psi_\lambda^{-1}\{[\frac{1}{6}],[\frac{2}{3}]\},\quad I_3=\psi_\lambda^{-1}\{[\frac{1}{3}],[\frac{5}{6}]\}.\]
	See Figure \ref{fig5} for an illustration of $\mcJ$ and $\mcJ^{(1)}$. On can see that $\mcJ$ is a perserved $\mathcal G$-relation.
	\begin{figure}[htp]
		\includegraphics[width=5cm]{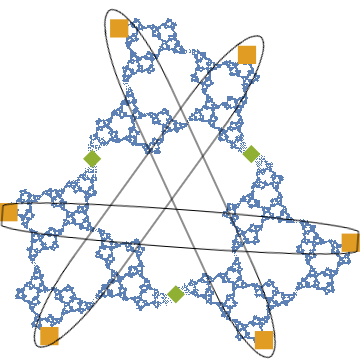}\hspace{1.5cm}
		\includegraphics[width=5cm]{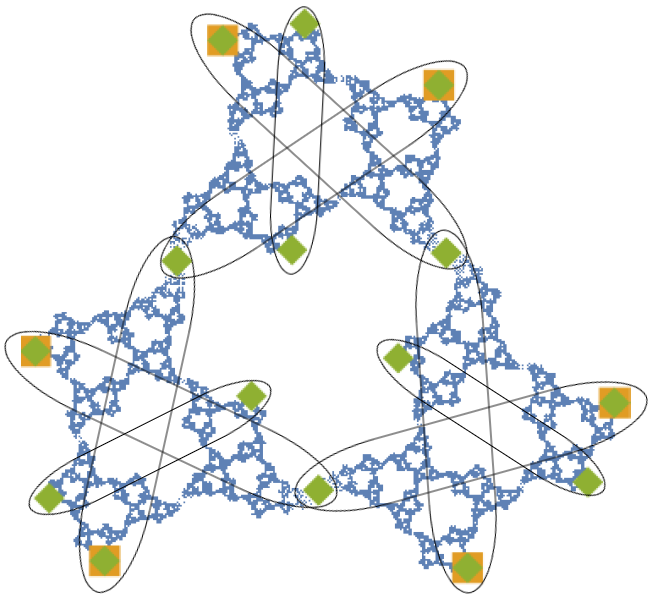}
		\caption{An illustration of a non-trivial preserved $\mathcal{G}$-relation $\mcJ$ and $\mcJ^{(1)}$.}\label{fig5}
	\end{figure}
	
\end{example}

In the subsequent lemmas, we provide a rough picture of all possible preserved $\mathcal{G}$-relations. The proof of the $m\geq 2$ case does not work here, but a similar argument still provides us some insight. Also, as $m+n$ and $n$ are coprime now, we can easily see that $\tilde{V}_0=V_0$. We substitute $m=1$ in the following {discussions}.

\begin{lemma}\label{lemma319}
	Let $\mcJ$ be a preserved relation on $V_0$ such that 
	\[x{\mcJ\mkern-10.5mu\backslash}^{(1)}y, \quad\forall x,y\in C\cup R_\lambda(C).\]
	Then we have $\mcJ=0$.  
\end{lemma}
\begin{proof}
The proof of the lemma is similar to that of Proposition \ref{prop317}. In fact, we will show the following claim analogous to Lemma \ref{lemma316}.\vspace{0.15cm}

\noindent\textit{Claim:} $d_{\beta_\lambda}(x,y)\geq \frac{1}{n+1}\Longrightarrow x{\mcJ\mkern-10.5mu\backslash} y$ for $x,y\in {V_0}$.
\vspace{0.15cm}

\textit{Proof of the claim.} We prove by contradiction. Clearly, $x,y$ belong to different $1$-cells, and by Lemma \ref{lemma311} (a) and Lemma \ref{lemma313} there is $1\leq i\leq n+1$ such that $x\in F_i{V}_0$, $y\in F_{i+1}{V}_0$ and  $x\mcJ^{(1)}c_i\mcJ^{(1)}y$. 
In addition, one can see that $d_{\beta_\lambda}(x,c_i)+d_{\beta_\lambda}(c_i,y)=d_{\beta_\lambda}(x,y)\geq\frac{1}{n+1}$. Without loss of generality, we assume $d_{\beta_\lambda}(x,c_i)\geq \frac{1}{2(n+1)}\geq\frac{1}{n(n+1)}$, which means that $x,c_i$ are not in the same $2$-cell. Noticing that $R_\lambda(x)\mcJ^{(1)}R_\lambda(c_i)$ by Lemma \ref{lemma311} (c), by Lemma \ref{lemma313}, there is a $G_\mcJ^{(1)}$-path connecting $R_\lambda(x),R_\lambda(c_i)$, and clearly the path will pass through some $c\in C$. Thus, $c\mcJ^{(1)} R_\lambda(c_i)$, which is a contradiction to the assumption.\vspace{0.15cm}

The lemma follows from the above claim and the argument underlying Proposition \ref{prop317}. 
\end{proof}

By Lemma \ref{lemma312} and Lemma \ref{lemma319}, we should have that the restriction of a non-trivial relation $\mcJ^{(1)}$  to $C\cup R_\lambda(C)$ is also non-trivial. This shows that  there are not too many non-trivial preserved $\mathcal{G}$-relations. To quantify this, we define two possible candidates.

\begin{definition}
Define $\kappa$ as the unique permutation on $\{1,2,\cdots,n+1\}$ such that 
\[R_\lambda(c_{\kappa(i)})\in F_i(K).\] 

(a). Define $\breve{G}_+=(C\cup R_\lambda(C),\breve{E}_+)$, with the edge set \[\breve{E}_+=\{\{c_i,R_\lambda(c_{\kappa(i)})\}:1\leq i\leq {n+1}\}.\]
Define $G_+=\bigcup_{k=1}^\infty R_\lambda^k(\breve{G}_+)$, and define {the} equivalence relation $\mcJ_+$ on $V_0$ {by}
\[x\mcJ_+ y \Longleftrightarrow x {\text{ and }} y\text{ belong to {the} same connected component of }G_+.\]
For $1\leq i\leq {n+1}$, let $I_{i,+}$ be the equivalence class of $\mcJ_+$ that contains $R_\lambda(c_{\kappa(i)})$.

(b). Define $\breve{G}_-=(C\cup R_\lambda(C),\breve{E}_-)$, with the edge set \[\breve{E}_-=\{{\{c_{i-1},R_\lambda(c_{\kappa(i)})\}}:1\leq i\leq {n+1}\}.\]
Define $G_-=\bigcup_{k=1}^\infty R_\lambda^k(\breve{G}_-)$, and define {the} equivalence relation $\mcJ_-$ on $V_0$ {by}
\[x\mcJ_- y \Longleftrightarrow x {\text{ and } } y\text{ belong to {the} same connected component of }G_-.\]
For $1\leq i\leq {n+1}$, let $I_{i,-}$ be the equivalence class of $\mcJ_-$ that contains $R_\lambda(c_{\kappa(i)})$.
\end{definition}

We shall see that $\mcJ_+$ and $\mcJ_-$ are the only  possibilities for non-trivial preserved $\mathcal{G}$-relations. 

\begin{lemma}\label{lemma321}
	Let $\mcJ$ be a non-trivial preserved $\mathcal{G}$-relation on $V_0$, then we have either $\mcJ=\mcJ_+$ or $\mcJ=\mcJ_-$. In addition,  we always have ${n+1}$ disjoint equivalence classes $I_i$ such that 
	\[R_\lambda(c_{\kappa(i)})\in I_i, \quad V_0=\bigcup_{i=1}^{n+1}I_i.\]
\end{lemma}
\begin{proof}
	According to Lemma \ref{lemma312}, Lemma \ref{lemma319} and Lemma \ref{lemma311} (c),  there {are $c_i\in C$ and $x\in R_\lambda(C)$} such that $c_i\mcJ^{(1)} x$. In addition, by Lemma \ref{lemma313}, we have $x\in F_iV_0$ or $x\in F_{i+1}V_0$. In other words, either $c_i\mcJ^{(1)} R_\lambda(c_{\kappa(i)})$ or $c_i\mcJ^{(1)} R_\lambda(c_{\kappa(i+1)})$. By the rotation symmetry, we can see that either $\breve{E}_+\subset \mcJ^{(1)}$ or $\breve{E}_-\subset \mcJ^{(1)}$, so 
	\[\mcJ_+\subset \mcJ\text{ or }\mcJ_-\subset \mcJ.\]
	
	Without loss of generality, we may assume $\mcJ_+\subset \mcJ$, and we need to show that $\mcJ_+=\mcJ$. For any $x\in V_0$, there exists $1\leq i\leq n+1$ and $k\geq 1$ such that $R_\lambda^k(c_{\kappa(i)})=x$. Then, we have 
	\[x=R_\lambda^k(c_{\kappa(i)})\mcJ_+ R_\lambda^{k-1}(c_i)\mcJ_+ R_\lambda^{k-2}(c_{\kappa^{-1}(i)})\mcJ_+\cdots\mcJ_+ R_\lambda(c_{\kappa^{2-k}(i)}),\]
	which implies that $x\in I_{\kappa^{1-k}(i),+}$. So
	\[V_0=\bigcup_{i=1}^{n+1}I_{i,+}.\]
	If $\mcJ_+\neq \mcJ$, then there exists an equivalence class $I$ of $\mcJ$ and $i\neq j$ such that 
	\[I_{i,+}\cup I_{j,+}\subset I.\]
	This implies that $R_\lambda(c_{\kappa(i)})\mcJ R_\lambda(c_{\kappa(j)})$, and thus $j=i+1$ (without loss of generality) and  $R_\lambda(c_{\kappa(i)})\mcJ^{(1)} c_i\mcJ^{(1)}R_\lambda(c_{\kappa(i+1)})$ by Lemma \ref{lemma313}. By rotation symmetry, we also have $R_\lambda(c_{\kappa_{i+1}})\mcJ^{(1)}c_{i+1}$. Thus, $c_i\mcJ^{(1)}c_{i+1}$, which is a contradiction by  Lemma \ref{lemma312}. So we have $\mcJ_+=\mcJ$.
	
	Lastly, we claim that the equivalence classes $I_{i,+},1\leq i\leq n+1$ are disjoint. Otherwise, by a same argument as above, we can see that $c_i\mcJ^{(1)}c_{i+1}$ for some $i$. 
\end{proof}

Next, we {roughly describe the structure} of $\mcJ^{(1)}$. 

\begin{lemma}\label{lemma322}
	Let $\mcJ$ be a non-trivial preserved $\mathcal{G}$-relation, and $I_i$'s be the equivalence classes of $\mcJ$ as in Lemma \ref{lemma321}. 
	
	(a). For $(i,i')\neq (j,j')\in \{1,\cdots,n+1\}^2$, we have 
	\[\begin{aligned}
	x\mcJ^{(1)} y,\forall x\in F_iI_{i'}, y\in F_jI_{j'}\Longleftrightarrow &j=i+1, i’=j'=\kappa^{-1}(i)\\&\text{ or }i=j+1, i'=j'=\kappa^{-1}(j).
	\end{aligned}
	\]
	In particular, we have that {$\mcJ^{(1)}$ consists of $n(n+1)$ equivalent classes.}
	
	(b). Define $I_i^{(1)}$ to be the equivalence class of $\mcJ^{(1)}$ that contains $I_i$. We have 
	\[I_i^{(1)}=
	\begin{cases}
	F_iI_{\kappa^{-1}(i)}\cup F_{i+1}I_{\kappa^{-1}(i)},\text{ if }\mcJ=\mcJ_+,\\
	F_iI_{\kappa^{-1}(i)}\cup F_{i-1}I_{\kappa^{-1}(i)},\text{ if }\mcJ=\mcJ_-.
	\end{cases}\]
\end{lemma}

\begin{proof}
	(a). `$\Longleftarrow$': Let $1\leq i\leq n+1$. Noticing that $R_\lambda(c_i)\in I_{\kappa^{-1}(i)}$ and $c_i=F_i(R_\lambda(c_i))=F_{i+1}(R_\lambda(c_i))$, we have  $c_i\in F_iI_{\kappa^{-1}(i)}\cap F_{i+1}I_{\kappa^{-1}(i)}${, so} $F_iI_{\kappa^{-1}(i)}$ and $F_{i+1}I_{\kappa^{-1}(i)}$ are subsets of a same $\mcJ^{(1)}$ class. 
	
	\noindent`$\Longrightarrow$': First, let's show that $i'=j'$. In fact, if the left side holds, then $R_\lambda(x)\mcJ R_\lambda(y),\forall x\in F_iI_{i'}, y\in F_jI_{j'}$ by Lemma \ref{lemma311} (b), which is only possible if $i'=j'$. 
	
	Next, we apply Lemma \ref{lemma313} to see that either $i=j$ or $|i-j|=1$. However, the former case is impossible since it  would imply $(i,i')=(j,j')$. Thus, $j=i+1$ or $i=j+1$. Without loss of generality, we assume $j=i+1$. Then, by Lemma \ref{lemma313}, we have $x\mcJ^{(1)} c_i\mcJ^{(1)}y$ for $x\in F_iI_{i'},y\in F_jI_{i'}$. Thus $R_\lambda(x)\mcJ R_\lambda(c_i)$, which implies that $i'=\kappa^{-1}(i)$.
		
	{From this conclusion it is} easy to see that there are $(n+1)^2-(n+1)$ equivalence classes of $\mcJ^{(1)}$, as there are $n+1$ pairs of sets of the form $F_iI_{i'}$ matched. 
	
	(b). We look at the case $\mcJ=\mcJ_+$ only, since the argument for $\mcJ=\mcJ_-$ is the same. In this case, we have $c_i\mcJ^{(1)}R_\lambda(c_{\kappa(i)})$ by definition{, so} $c_i\in I_i^{(1)}$, which implies 
	\[F_iI_{\kappa^{-1}(i)}\cup F_{i+1}I_{\kappa^{-1}(i)}\subset I_i^{(1)}.\]
	Clearly, the left side itself is an equivalence class of $\mcJ^{(1)}$ by (a). 
\end{proof}

 We can now prove the existence of a $\mathcal{G}$-symmetric form on $K_\lambda$ when $n\geq 2$, $m=1$. 

\begin{proposition}\label{prop323}
	Let $R_\lambda(z)=z^n+\frac{\lambda}{z}$ with $n\geq 2$ be a MS map. Let $\mcJ$ be a non-trivial preserved $\mathcal{G}$-relation on $V_0$. Then we have 
	
	(a). $\overline{\rho}^{\mathcal G}_\mcJ\leq 1$.
	
	(b). ${\underline{\rho}^{\mathcal G}_{V_0/\mcJ}}\geq 1+\frac{1}{n}>1$.
	
	\noindent In particular, there exists exactly one $\mathcal G$-symmetric solution to (\ref{eqn32}).
\end{proposition}

\begin{proof} (a).  We abbreviate $x_i=R_\lambda(c_{\kappa(i)})$ below. Let $I_i$ be the equivalence class of $\mcJ$  containing $x_i$ as in Lemma \ref{lemma321}. Define a form $\mathcal D\in\mathcal M_{\mcJ}$ as
	\[\mcD(f)=\sum_{i=1}^{n+1} \sum_{x\in I_i\setminus \{x_i\}}\big(f(x)-f(x_i)\big)^2, \quad\forall f\in l(V_0).\]
	Now, for each $f\in l(V_0)$, we define the extension $f_1\in l(V_1)$ of $f$  to be
	\[f_1(x)=\begin{cases}
	f(x),&\text{ if }x\in V_0,\\
	f(x_i),&\text{ if }x\in I_i^{(1)}\setminus I_i,\\
	0,&\text{ if }x\in V_1\setminus\bigcup_{i=1}^{n+1} I_i^{(1)}. 
	\end{cases}\] 
	If $\mcJ=\mcJ_+$, we have
	\[\begin{aligned}
	\mathcal{D}^{(1)}(f_1)&=\sum_{i=1}^{n+1}\sum_{j=1}^{n+1}\sum_{x\in I_{j}\setminus\{x_{j}\}}\big(f_1(F_ix)-f_1(F_ix_{j})\big)^2\\
	&=\sum_{i=1}^{n+1}\sum_{x\in I_{\kappa^{-1}(i)}}\big(f_1(F_ix)-f_1(c_i)\big)^2+\sum_{i=1}^{n+1}\sum_{x\in I_{\kappa^{-1}(i-1)}}\big(f_1(F_ix)-f_1(c_{i-1})\big)^2\\
	&=\sum_{i=1}^{n+1}\sum_{x\in I^{(1)}_i\setminus \{c_i\}} \big(f_1(x)-f_1(c_i)\big)^2=\sum_{i=1}^{n+1}\sum_{x\in I_i\setminus\{x_i\}}\big(f(x)-f(x_i)\big)^2=\mcD(f). 
	\end{aligned}\]
	A similar equation holds for the case $\mcJ=\mcJ_-$. Using the above equation we have
	\[\overline{\rho}^{\mathcal{G}}_\mcJ\leq \overline{\rho}^{\mathcal{G}}_\mcJ(\mcD)=\sup_{f\in l(V_0)\setminus l(V_0/\mcJ)}\frac{T_\mcJ\mcD(f)}{\mcD(f)}\leq \sup_{f\in l(V_0)\setminus l(V_0/\mcJ)}\frac{\mcD^{(1)}(f_1)}{\mcD(f)}=1.\] 
	
	(b). We define a form $\mathcal D\in \mathcal M_{V_0/\mcJ}$ as 
	\[\mcD(f)=\sum_{i=1}^{n+1}\big(f(I_{i+1})-f(I_{i})\big)^2,\quad \forall f\in l(V_0/\mcJ).\]
	Recall that we identify $l(V_0/\mcJ)$ with the subspace of $l(V_0)$ consisting of functions with constant value on each $I_i$, and identify $l(V_1/\mcJ^{(1)})$ with a subspace of $l(V_1)$ analogously. Let $f_1$ be an extension of $f$ to $l(V_1/\mcJ^{(1)})$, we have 
	\[\begin{aligned}
	\mcD^{(1)}(f_1)&=\sum_{i=1}^{n+1}\sum_{j=1}^{n+1}\big(f_1(F_iI_{j+1})-f_1(F_iI_{j})\big)^2\\
	&=\sum_{i=1}^{n+1}\Big(f_1\big(F_iI_{\kappa^{-1}(i-1)}\big)-f_1\big(F_iI_{\kappa^{-1}(i)}\big)\Big)^2+\sum_{i=1}^{n+1}\sum_{j\neq \kappa^{-1}(i-1)-1}\Big(f_1\big(F_iI_{j+1}\big)-f_1\big(F_iI_{j}\big)\Big)^2\\
	&\geq (1+\frac{1}{n})\sum_{i=1}^{n+1}\big(f(I_{i+1})-f(I_i)\big)^2=(1+\frac{1}{n})\mcD(f),
	\end{aligned}\]
	where we use the property that $\kappa^{-1}(i-1)-1=\kappa^{-1}(i)$ (in cyclic notation $n+1=0$) in the second equality,  and Lemma \ref{lemma322} (b)  together with an  effective resistance computation in the inequality. Thus, we have 
	\[\underline{\rho}^{\mathcal{G}}_{V_0/\mcJ}\geq \underline{\rho}^{\mathcal{G}}_{V_0/\mcJ}(\mcD)=\inf_{f\in l(V_0/\mcJ)}\frac{T_\mcJ\mcD(f)}{\mcD(f)}\geq 1+\frac1n.\]
	
	Finally, the existence and uniqueness of a $\mathcal{G}$-symmetric solution of (\ref{eqn32}) follows from Sabot's theorem, Theorem \ref{thm38} (b). 
\end{proof}

\subsection{Proof of uniqueness}
In this subsection, we return to the general MS Julia sets and prove that the $\mathcal{G}$-symmetric solution, which has been shown to exist, is the unique solution (without assuming $\mathcal G$-symmetry a priori). For this aim, we will consider the preserved relations $\mcJ$ on $V_0$, which are not assumed to be  $\mathcal{G}$-symmetric. Luckily, we can take the advantage of the existence of a symmetric form {and the} `ring' structure of the {level-1} cells. 

Throughout this subsection, we would admit the following settings:

We fix a solution $\mcD$ to (\ref{eqn32}) which has already been proved to exist. In particular, we have a positive constant $\eta$ (which is uniquely determined by the equation (\ref{eqn32})), such that 
\[T\mcD=\eta^{-1}\mcD.\]
Note that $0<\eta^{-1}<1$ by a well-known theorem (see \cite{ki2,ki3}), and the pair $(\mcD,\eta^{-1})$ is called a \textit{regular harmonic structure} on $K_\lambda$, which will generate a local resistance form $(\mcE,\mathcal{F})$ on $K_\lambda$. {For reference,} a standard proof {is available} in Kigami's book \cite{ki3}. \vspace{0.15cm}

We will show that 
\[\overline{\rho}_{\mcJ,k}<\underline{\rho}_{V_0/\mcJ,k}\]
for some $k\geq 1$. In particular, we will construct forms based on the solution $\mcD$. 

Recall{ ing} the definition of $\mcD_{V_0/\mcJ}$, the following easy lemma follows from Sabot's paper \cite{Sabot}.

\begin{lemma}\label{lemma324}
Let $\mcJ$ be a preserved relation on $V_0$, and let $f\in l(V_0/\mcJ)$. We have 
\[T_{V_0/\mcJ}\mcD_{V_0/\mcJ}(f)\geq T\mcD(f)=\eta^{-1}\mcD(f).\]
\end{lemma}

As an easy consequence, we can see 
\begin{lemma}\label{lemma325}
	Let $\mcJ$ be a preserved relation on $V_0$, then $\underline{\rho}_{V_0/\mcJ,k}\geq \eta^{-k}$ for any $k\geq 1$. 
\end{lemma}

We will devote the rest of this subsection {to demonstrate a corresponding statement:}
\begin{proposition}\label{prop326}
Let $\mcJ$ be a preserved relation on $V_0$, then $\overline{\rho}_{\mcJ,k}<\eta^{-k}$ for some $k\geq 1$. 
\end{proposition}

To prove the proposition, we will study the following  form in $\mathcal{M}_\mcJ$.
\begin{definition}\label{def327}
Let $\mcJ$ be a preserved relation on $V_0$, with equivalence classes $I_a$, $a=1,2,\cdots,N$.

(a).  We define 
	\[\mcD_\mcJ=\sum_{a=1}^N\mcD|_{I_a}.\]

(b). For $1\leq a\leq N$, we define 
\[\ell(I_a)=\{f\in l(V_0):f|_{V_0\setminus I_a}\equiv 0\}.\]

(c). For $1\leq a\leq N$ and $f\in \ell(I_a)$, we define $H^{(1)}_{\mcD_\mcJ,I_a}f$ {to be} the unique function in $l(V_1)$ such that 
\[(H^{(1)}_{\mcD_\mcJ,I_a}f)|_{V_1\setminus I_a^{(1)}}\equiv 0,\]
and 
\[\mcD_\mcJ^{(1)}\big(H^{(1)}_{\mcD_\mcJ,I_a}f\big)=T_\mcJ\mcD_\mcJ(f),\]
where we recall that $\mcD^{(1)}_\mcJ(g)=\sum_{i=1}^{m+n}\mcD_\mcJ(g\circ F_i)$ for any $g\in l(V_1)$.
\end{definition}

Since we no longer {are assuming rotational symmetry}, our starting points will be  Lemmas \ref{lemma311} and \ref{lemma312}. The following statement is an easy consequence of Lemma \ref{lemma311}.

\begin{lemma}\label{lemma328}
Let $\mcJ$ be a preserved relation on $V_0$, with equivalence classes $I_a,a=1,2,\cdots,N$. For any $1\leq a\leq N$ and $k\geq 1$, there is a unique $\mcJ^{(k)}$ class $I_a^{(k)}$ such that $I_a\subset I_a^{(k)}$. In addition, {the classes} $I_a^{(k)},a=1,2,\cdots,N$ are disjoint. {Finally, f}or any finite word $w$ with $|w|=k$, if $F_wV_0\cap I^{({k})}_a\neq \emptyset$, we have a unique $a_w\in \{1,2,\cdots,N\}$ such that $F_wI_{a_w}\subset I_a^{({k})}$. 
\end{lemma}
\begin{proof}
The existence of $I_a^{(k)}$ and the claim that $I_a^{(k)}$ are disjoint  follow quickly from Lemma \ref{lemma311} (a). To see the {last} statement, we assume there exist $I_b\neq I_{b'}$ such that $F_wI_b\subset I_a^{(k)}, F_wI_{b'}\subset I_a^{(k)}$. Then for $x\in I_b$ and $y\in I_{b'}$, using Lemma \ref{lemma311} (b), we have 
\[F_wx\mcJ^{(k)}F_wy\Longrightarrow x\mcJ y.\]
This is impossible {because $x$ and $y$ have been} chosen from different equivalence classes of $\mcJ$.  
\end{proof}

The above lemma allows us to prove the following result.
\begin{lemma}\label{lemma329}
Let $\mcJ$ be a preserved relation on $V_0$, with equivalence classes $I_a,a=1,2,\cdots,N$.

(a). Let $a\neq b$. Then 
\[\tilde{\mcD}(f,g)=0, \quad\forall\tilde{\mcD}\in \mathcal{M}_\mcJ, f\in \ell(I_a),g\in\ell(I_b).\]

(b). Let $k\geq 1$ and $f\in  l(V_0)$, we have  
\begin{equation}\label{eqn35}
\eta^kT_\mcJ^k\mcD_\mcJ(f)\leq\eta^{k-1}T_\mcJ^{k-1}\mcD_\mcJ(f)\leq\cdots\leq\mcD_\mcJ(f).
\end{equation}
Moreover,  for fixed $1\leq a\leq N$, let $h_f:=H_{\mcE,I_a}(f)$, which is the harmonic extension of $f|_{I_a}$ with respect to the form $(\mathcal{E},\mathcal{F})$ generated by $\mcD$. Also, let $W_{k,a}=\{|w|=k:F_wV_0\cap I_a^{(k)}\neq \emptyset\}$. Then, 
\begin{equation}\label{eqn36}
\eta^kT_\mcJ^k\mcD_\mcJ(f)=\mcD_\mcJ(f)
\end{equation}
only if 
\begin{equation}\label{eqn37}
h_f\circ F_w=\begin{cases}
H_{\mcE,I_{a_w}}\big((h_f\circ F_w)\cdot 1_{I_{a_w}}\big), &\text{ if } w\in \bigcup_{l=1}^k W_{l,a},\\
constant,&\text{ if } 1\leq |w|\leq k, w\notin\bigcup_{l=1}^k W_{l,a}.
\end{cases}
\end{equation}
where $I_{a_w}$ is the unique $\mcJ$ class such that $F_wI_{a_w}\subset I_a^{(|w|)}$ as shown in Lemma \ref{lemma328}.
\end{lemma}
\begin{proof}
(a) is obvious. We will focus on (b). Fix $1\leq a\leq N$ and $f\in \ell(I_a)$, then we can show the $k=1$ case of (\ref{eqn35}) by the following computation,
\begin{equation}\label{eqn38}
\begin{aligned}
\eta T_\mcJ\mcD_\mcJ(f)&=\eta\mcD^{(1)}_\mcJ(H^{(1)}_{\mcD_\mcJ,{I_a}}f)\\
&\leq \eta\mcD^{(1)}_\mcJ(h_f\cdot 1_{I^{(1)}_a})=\eta\sum_{i=1}^{m+n} \mcD_\mcJ\big((h_f\cdot 1_{I^{(1)}_a})\circ F_i\big)\\
&=\eta\sum_{i\in W_{1,a}} \mcD_\mcJ\big((h_f\circ F_i)\cdot 1_{a_i}\big)\\
&\leq_{(*1)} \eta\sum_{i=1}^{m+n} \mcD(h_f\circ F_i)=\mcD(h_f|_{V_0})=\mcD|_{I_a}(f)=\mcD_\mcJ(f),
\end{aligned}
\end{equation}
where in the first inequality we use the fact that $H^{(1)}_{\mcD_\mcJ,{I_{a}}}f$ is the minimal energy extension, and in the second inequality we use Lemma \ref{lemma328}. For a general $f\in l(V_0)$, we apply (a) to show (\ref{eqn35}) still holds,
\[\begin{aligned}
\eta T_\mcJ\mcD_\mcJ(f)=\eta\sum_{a=1}^NT_\mcJ\mcD_{\mcJ}(f\cdot 1_{I_a})\leq \sum_{a=1}^N\mcD_{\mcJ}(f\cdot 1_{I_a})=\mcD_\mcJ(f),
\end{aligned}\]  
since $\ell(I_a)$'s are pairwise orthogonal for different $a$'s with respect to forms in $\mathcal M_\mcJ$ by (a). 

Next, we return to study {the conditions under which} (\ref{eqn35}) becomes an equality. Clearly, for (\ref{eqn38}) to be an equality, we need `$\leq_{(*1)}$' holds, which is equivalent to the condition (\ref{eqn37}) for $k=1$. To extend the observation to general $k\geq 1$, we  induct. 

Assuming that (\ref{eqn36}) holds, we have immediately that $\eta T_\mcJ\mcD_\mcJ(f)=\mcD_\mcJ(f)$ by (\ref{eqn35}). So (\ref{eqn37}) holds immediately for $|w|=1$. Next, we apply the induct{ive} assumption to get the following inequality, 
\[
\begin{aligned}
\eta^kT^k_\mcJ\mcD_\mcJ(f)&=\eta^k(T^{k-1}_\mcJ\mcD_\mcJ)^{(1)}(H^{(1)}_{T_\mcJ^{k-1}\mcD_\mcJ,{I_a}}f)\\
&\leq \eta^k(T^{k-1}_\mcJ\mcD_\mcJ)^{(1)}(h_f\cdot 1_{I^{(1)}_a})\\
&=\eta\sum_{i\in W_{1,a}} \eta^{k-1}T^{k-1}_\mcJ\mcD_\mcJ\big((h_f\cdot 1_{I^{(1)}_a})\circ F_i\big)\\
&\leq_{(*2)} \eta\sum_{i\in W_{1,a}} \mcD_\mcJ\big((h_f\cdot 1_{I^{(1)}_a})\circ F_i\big)\\
&=\eta\sum_{i\in W_{1,a}} \mcD\big(h_f\circ F_i\big)=\mcD(h_f|_{V_0})=\mcD|_{I_a}(f)=\mcD_\mcJ(f),
\end{aligned}
\]
where we use the fact $h_f\circ F_i=H_{\mcE,I_{a_i}}\big((h_f\circ F_i)\cdot 1_{I_{a_i}}\big)=H_{\mcE,I_{a_i}}\big((h_f\cdot 1_{I_a^{(1)}})\circ F_i\big)$ by (\ref{eqn37}) in the last line.

Clearly, when (\ref{eqn36}) holds, we {require} `$\leq_{(*2)}$' to be an equality. {This means}
\[\eta^{k-1}T^{k-1}_\mcJ\mcD_\mcJ\big((h_f\cdot 1_{I^{(1)}_a})\circ F_i\big)=\mcD_\mcJ\big((h_f\cdot 1_{I^{(1)}_a})\circ F_i\big), \quad\forall i\in W_{1,a}.\]
By  the inductive hypothesis, this implies that (\ref{eqn37}) holds for  $1\leq |w|\leq k-1$, with $h_f\circ F_i$ replacing $h_f$ and $a_i$ replacing $a$ for $i\in W_{1,a}$, and thus (\ref{eqn37}) holds for any $2\leq |w|\leq k$. 
\end{proof}

To understand the condition (\ref{eqn37}) better, we introduce the concept of flows. The concept is also known as  that of normal derivatives for its role in the Gauss-Green's formula, but we prefer the word `flow' for  the more intuitive physical picture.

\begin{definition}
Let $h$ be a harmonic function on  $K_\lambda$, i.e. $\mathcal{E}(h)=\mcD(h|_{V_0})$. 

(a). For $x\in V_0$, we define the flow (normal derivative) of $h$ at $x$ as $\partial_n h(x)=\mcD(h,1_{x})$.  We write $V_{0,h}:=\{x\in V_0:\partial_n h(x)\neq 0\}$ for the set of vertices with nonzero flows.

(b). For $c_i\in C$, we say $h$ has nonzero flow passing through $c_i$ if 
\[\partial_n (h\circ F_j)\big(R_\lambda(c_i)\big)\neq 0, \text{ for }j=i,i+1.\] 
We write $C_h:=\{c_i\in C:\partial_n (h\circ F_i)\big(R_\lambda(c_i)\big)\neq 0\}$ for the set of critical points with nonzero flows.
\end{definition}

We enumerate some simple properties of  flow.

\noindent(P1). $\sum_{x\in V_0}\partial_n h(x)=0$.

\noindent(P2). $\partial_n (h\circ F_i)\big(R_\lambda(c_i)\big)+\partial_n (h\circ F_{i+1})\big(R_\lambda(c_i)\big)=0$.

\noindent(P3). For $x\in V_0\cap F_i(V_0)$, we have the scaling identity $\partial_n h(x)=\eta\partial_n (h\circ F_i)\big(R_\lambda(x)\big)$.

The following observation is obvious. 
\begin{lemma}\label{lemma331}
Let $\mcJ$ be a preserved relation on $V_0$, with equivalence classes $I_a,a=1,2,\cdots,N$. Let $h$ be a harmonic function on $K_\lambda$ and $1\leq a\leq N$. If (\ref{eqn37}) holds for $k=1$ with $h$ replacing $h_f$, we have $C_h\subset I_a^{(1)}$, which implies that $C_h\neq C$.
\end{lemma}

Next, we prove that $T_\mcJ^k\mcD_\mcJ<\mcD_\mcJ$ for some $k\geq1$ by means of Lemmas \ref{lemma329} and \ref{lemma331}. First, {we} show an estimate for a single $f\in l(V_0)$.  

\begin{lemma}\label{lemma332}
Let $\mcJ$ be preserved relation on $V_0$, and $f\in l(V_0)\setminus l(V_0/\mcJ)$. Then there exists $k\geq 1$ such that 
\[\eta^kT_\mcJ^k\mcD_\mcJ(f)<\mcD_\mcJ(f).\]
\end{lemma}
\begin{proof}
{We begin by taking} $f\in \ell(I_a)$ for some $1\leq a\leq N$ which is non-constant on $I_a$. We will prove the lemma by contradiction.
Assume $\eta^kT_\mcJ^k\mcD_\mcJ(f)=\mcD_\mcJ(f),\forall k\geq 0$. Then, by {Lemmas} \ref{lemma329} and \ref{lemma331}, we have 
\begin{equation}\label{eqn39}
\# C_{h_f\circ F_w}<m+n, 
\end{equation}
for any finite word $w$ and $h_f=H_{\mathcal{E},I_a}(f)$. We will see this is impossible for some {word} $w$. We will {construct such a word in two steps.}\vspace{0.15cm} 

\noindent\textbf{Step 1.} \textit{We can find a finite word $w$ such that $\# V_{0,h_f\circ F_w}=2$. }
\vspace{0.15cm}

To achieve this, we will construct a sequence of finite words $w_k$ of length $k$ inductively. For convenience, we write $N_k:=\# V_{0,h_f\circ F_{w_k}}$ and $M_k:=\# C_{h_f\circ F_{w_k}}$. We start from $w_0=\emptyset$,  and choose $w_1,w_2,\cdots$ in accordance with the following rules,  stopping when $N_k=2$. 

Assume we have chosen $w_k$ with $N_k\geq 3$, we will choose $w_{k+1}=w_k i$ with $1\leq i\leq m+n$ following two possible cases.

\textit{Case 1.1: $M_k=0$.} In this case, we simply choose an $1\leq i\leq m+n$ such that $\#V_{0,h_f\circ F_{w_ki}}>0$. Let $w_{k+1}=w_k i$ and  clearly we have $\#N_{k+1}\leq \# N_k$.

\textit{Case 1.2: $1\leq M_k\leq m+n-1$.} In this case, we have at least $M_k+1$ different $i$'s such that $\#V_{0,h_f\circ F_{w_ki}}>0$. Moreover,
\[\sum_{i=1}^{m+n} \#V_{0,h_f\circ F_{w_ki}}=2M_k+N_k. \]
Thus, we can choose $w_{k+1}=w_ki$ such that 
\begin{equation}\label{eqn310}
\#N_{k+1}\leq \frac{N_k}{M_k+1}+\frac{2M_k}{M_k+1}<N_k,
\end{equation}
where the last `$<$' holds since we always have $N_k\geq 3$ (otherwise, we will stop the construction) and $M_k\geq 1$. 

Continuing the construction, we can easily see that $N_0\geq N_1\geq N_2\geq \cdots$. However, Case 1.1 cannot repeat consecutively for infinitely many  iterations, otherwise, since it does not introduce any new flow, after long iterations we will have a small cell $F_{w}K_\lambda$ which contains only $1$ original flow at its boundary, which is impossible by (P1). Each time we face Case {1.2}, we have a strict decrease in  $N_k$.  Eventually, we can find a $k\geq 1$ such that $\#N_k=2$. \vspace{0.15cm}

\noindent\textbf{Step 2.} \textit{We can find a finite word $w$ such that $V_{0,h_f\circ F_w}=\{x,y\}$, with $x\in F_iK_\lambda$, $y\in F_jK_\lambda$ {and} $i\neq j$.}
\vspace{0.15cm}

Let $w$ be the word found in Step 1. If $w$ satisfies the condition, we are done. Otherwise, we have $V_{0,h_f\circ F_w}\subset F_iK_\lambda$ for some $1\leq i\leq m+n$. We may face  either

\textit{Case 2.1: $C_{h_f\circ F_w}=\emptyset$,} or \textit{Case 2.2: $C_{h_f\circ F_w}\neq\emptyset$.}

{Case 2.2 is clearly impossible} since that would imply $C=C_{h_f\circ F_w}$ by the ring structure of $1$-cells of $K_\lambda$. {Hence,} only Case 2.1 is possible, and we {may} choose $w'=wi$. After repeating the above argument finitely many times, we will find a finite word $w''$ satisfying the desired condition.\vspace{0.15cm}

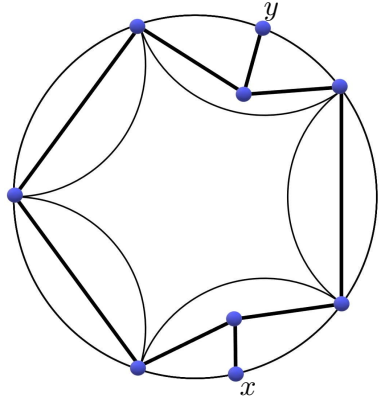
\begin{figure}[htp]
	\includegraphics[width=5cm]{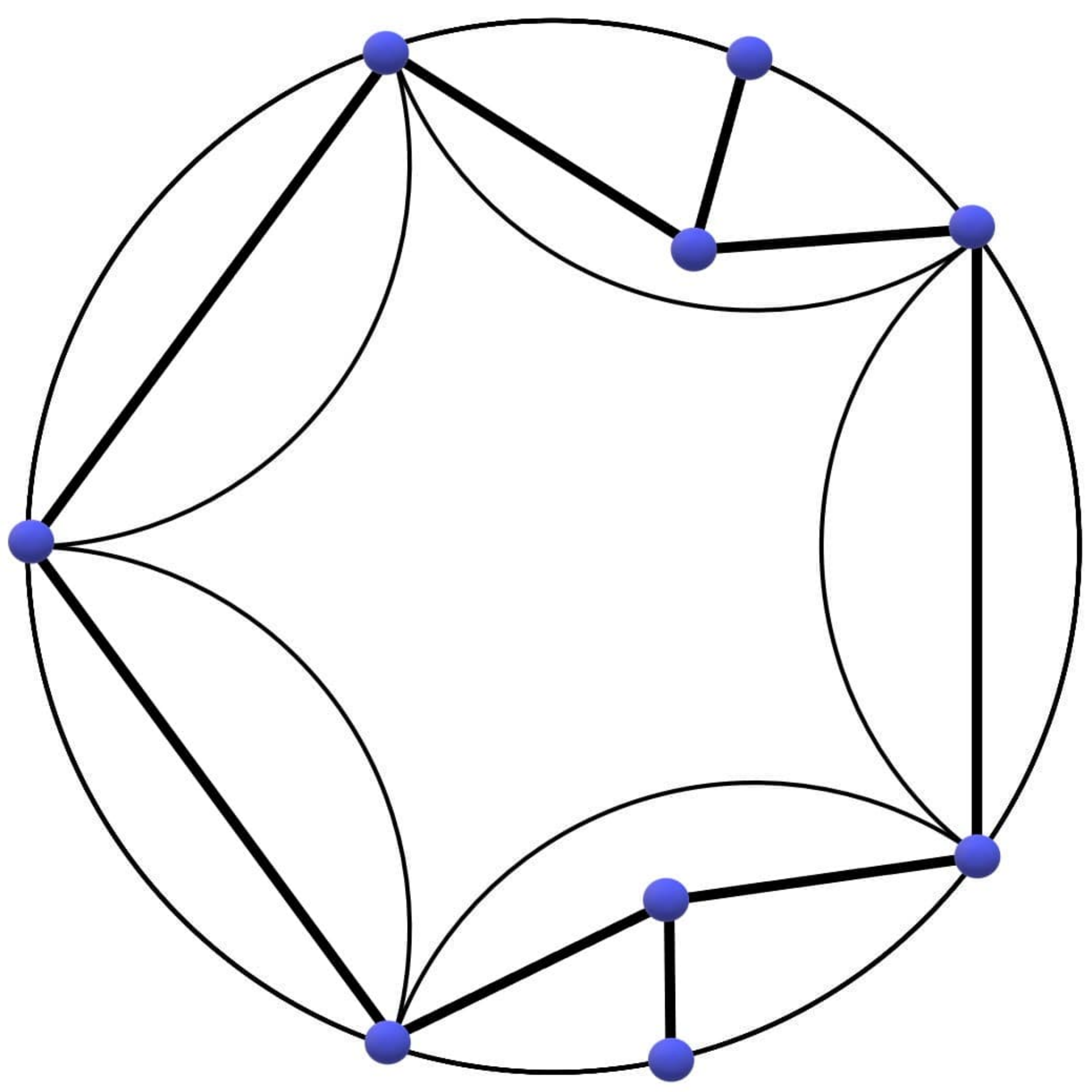}
	\begin{picture}(0,0)
	\put(-57,-4){$x$}
	\put(-48,140){$y$}
	\end{picture}
	\caption{The restriction of $\mcE$ onto $C\cup\{x,y\}$.}\label{fig7}
\end{figure}

Now, we look at the word $w$ chosen in Step 2. There are $x,y\in V_0$ such that $h_f\circ F_w$ is harmonic in $K_\lambda\setminus\{x,y\}$. In particular, by  restricting $\mcE$ to $\{x,y\}\cup C$ and  applying the $\Delta-Y$ transformation ({see books \cite{ki3,s3} for the formulas of $\Delta-Y$ transformation}) in the cells $F_iK_\lambda, F_jK_\lambda$ (see Figure \ref{fig7} for an illustration of the restricted electrical network), we can easily see that $h_f\circ F_w$ has nonzero flow at each $c\in C$. This contradicts (\ref{eqn39}).

Finally, for a general $f\in l(V_0)\setminus l(V_0/\mcJ)$, there is at least one $1\leq a\leq N$ such that $f\cdot 1_{I_a}\in \ell (I_a)$ is non-constant on $I_a$. So there is $k\geq 1$ such that $\eta^kT^k_\mcJ\mcD_\mcJ(f\cdot 1_{I_a})<\mcD_\mcJ(f\cdot 1_{I_a})$, and then,
\[\eta^kT^k_\mcJ\mcD_\mcJ(f)=\sum_{1\leq a'\leq N}\eta^kT^k_\mcJ\mcD_\mcJ(f\cdot 1_{I_{a'}})<\sum_{1\leq a'\leq N}\mcD_\mcJ(f\cdot 1_{I_{a'}})=\mcD_\mcJ(f),\]
which finishes the proof.
\end{proof}

We are now ready to prove Proposition \ref{prop326}.
\begin{proof}[Proof of Proposition \ref{prop326}]
	Define
	\[B_\mcJ=\{g\in l(V_0):
	\max \{|g(x)-g(y)|:x\mcJ y\}=1,\text{ and }
	\|g\|_{l^\infty(V_0)}=1/2
	\}
	\]
	as a compact subset of $l^\infty(V_0)$. Then for any fixed $f\in l(V_0)\setminus l(V_0/\mcJ)$, we can find $g\in B_\mcJ\cap \{cf+u:c\in \mathbb{R},u\in l(V_0/\mcJ)\}$, and it is easy to see that 
	\[\frac{T_\mcJ^k\mcD_\mcJ(f)}{\mcD_\mcJ(f)}=\frac{T_\mcJ^k\mcD_\mcJ(g)}{\mcD_\mcJ(g)},\quad\forall k\geq 1.\]
	In addition, since $\mcD_\mcJ(g)>0$ on $B_\mcJ$, the function $\frac{T_\mcJ^k\mcD_\mcJ(g)}{\mcD_\mcJ(g)}$ is continuous on $B_\mcJ$. 
	
	We now claim that $\eta^kT_\mcJ^k\mcD_\mcJ<\mcD_\mcJ$ for some $k\geq 1$. Assume {not}, then for any $k\geq 1$, there exists $f_k\in l(V_0)\setminus l(V_0/\mcJ)$ such that $\eta^kT_\mcJ^k\mcD_\mcJ(f_k)=\mcD_\mcJ(f_k)$. In addition, we can require that $f_k\in B_\mcJ$ by the previous argument. Thus, there exists a subsequence $k_l,l\geq 1$ such that $f_{k_l}$ converges to a function $f\in B_\mcJ$. Clearly,
	\[\eta^kT_\mcJ^k\mcD_\mcJ(f)=\lim\limits_{l\to\infty}\eta^kT_\mcJ^k\mcD_\mcJ(f_{k_l})=\lim\limits_{l\to\infty}\mcD_\mcJ(f_{k_l})=\mcD(f), \quad\forall k\geq 1.\]
	This contradicts Lemma \ref{lemma332}.
	
	Lastly, the proposition follows from the inequality
	\[\overline{\rho}_{\mcJ,k}\leq \overline{\rho}_{\mcJ,k}(\mcD)=\sup_{f\in l(V_0)\setminus l(V_0/\mcJ)}\frac{T_\mcJ^k\mcD_\mcJ(f)}{\mcD_\mcJ(f)}=\sup_{f\in B_\mcJ}\frac{T_\mcJ^k\mcD_\mcJ(f)}{\mcD_\mcJ(f)}<\eta^{-k}.\]
\end{proof}

Finally, we conclude the proof  of {the} main result in this section. 
\begin{proof}[Proof of Theorem \ref{thm32}]
	The existence of a symmetric form follows from Proposition \ref{prop317} and Proposition \ref{prop323}. The uniqueness follows from Lemma \ref{lemma325}, Proposition \ref{prop326}, Theorem \ref{thm38} and the remark after Theorem \ref{thm38}.
\end{proof}

\subsection{Examples} 
The MS Julia sets can be quite complicated in general. We are only able to  compute the exact forms for some simple examples. 

\begin{example}
	Consider $R_\lambda(z)=z^n+\frac{\lambda}{z^m}$ with $\theta_\lambda=\frac{l}{n(m+n)}$ for $m\geq 1$, $n\geq 2$ and $1\leq l\leq n-1$. {For $m,n$ fixed, these parameters correspond to the Julia sets with the smallest possible }$\tilde{V_0}=\{\psi_\lambda^{-1}[\frac{k}{m+n}]:1\leq k\leq m+n-1\}$. See Figure \ref{fig8} for {examples of such sets.} 
	
	\begin{figure}[htp]
		\includegraphics[width=4.7cm]{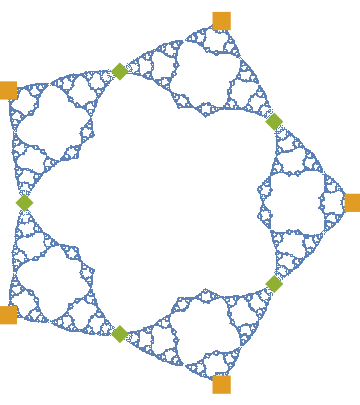}\qquad
		\includegraphics[width=5cm]{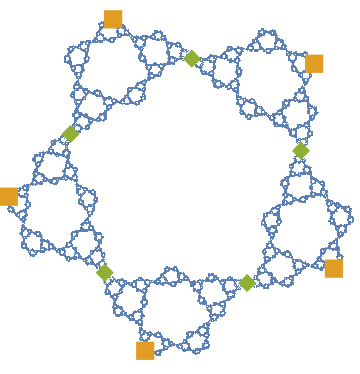}
		\begin{picture}(0,0)
		\put(-305,-8){$m=3,n=2$\emph{ and }$\theta_\lambda=\frac{1}{10}$}
		\put(-130,-8){$m=2,n=3$\emph{ and }$\theta_\lambda=\frac{2}{15}$}
		\end{picture}
		\caption{Some MS Julia sets with smallest $\tilde{V}_0$.}\label{fig8}
	\end{figure}
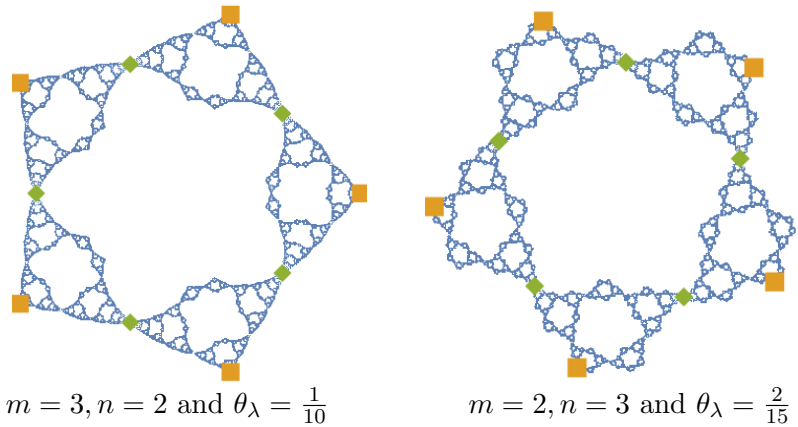
	
	{Due to} Theorem \ref{thm32}, there exists exactly one balanced form on $K_\lambda$. To simplify the calculation, we consider $V'_0:=\{p_0,p_1,p_2\}$ with 
	\[\psi_\lambda(p_0)=[0],\quad\psi_\lambda(p_1)=[\frac{l}{m+n}],\quad\psi_\lambda(p_2)=[\frac{m+l}{m+n}],\]
	and set $V'_1=C\cup \tilde{V}_0$. By a simple computation (using the $\Delta-Y$ transformation to restrict the form on $V_1'$ to $V_0'$), we get the exact value of the renormalization constant $\eta$ to be
	\[\eta=\frac{1}{2}+\frac{mn}{2(m+n)}+\frac{1}{2}\sqrt{(\frac{mn}{m+n}-1)^2+\frac{8l(n-l)}{m+n}},\] 
	with the form $\mcD|_{V_0'}$ as shown in Figure \ref{fig9}. {We omit the computation here, and readers can find a similar computation for the pentagasket in the book \cite{s3}.}
	
	\begin{figure}[htp]
	\includegraphics[width=3.5cm]{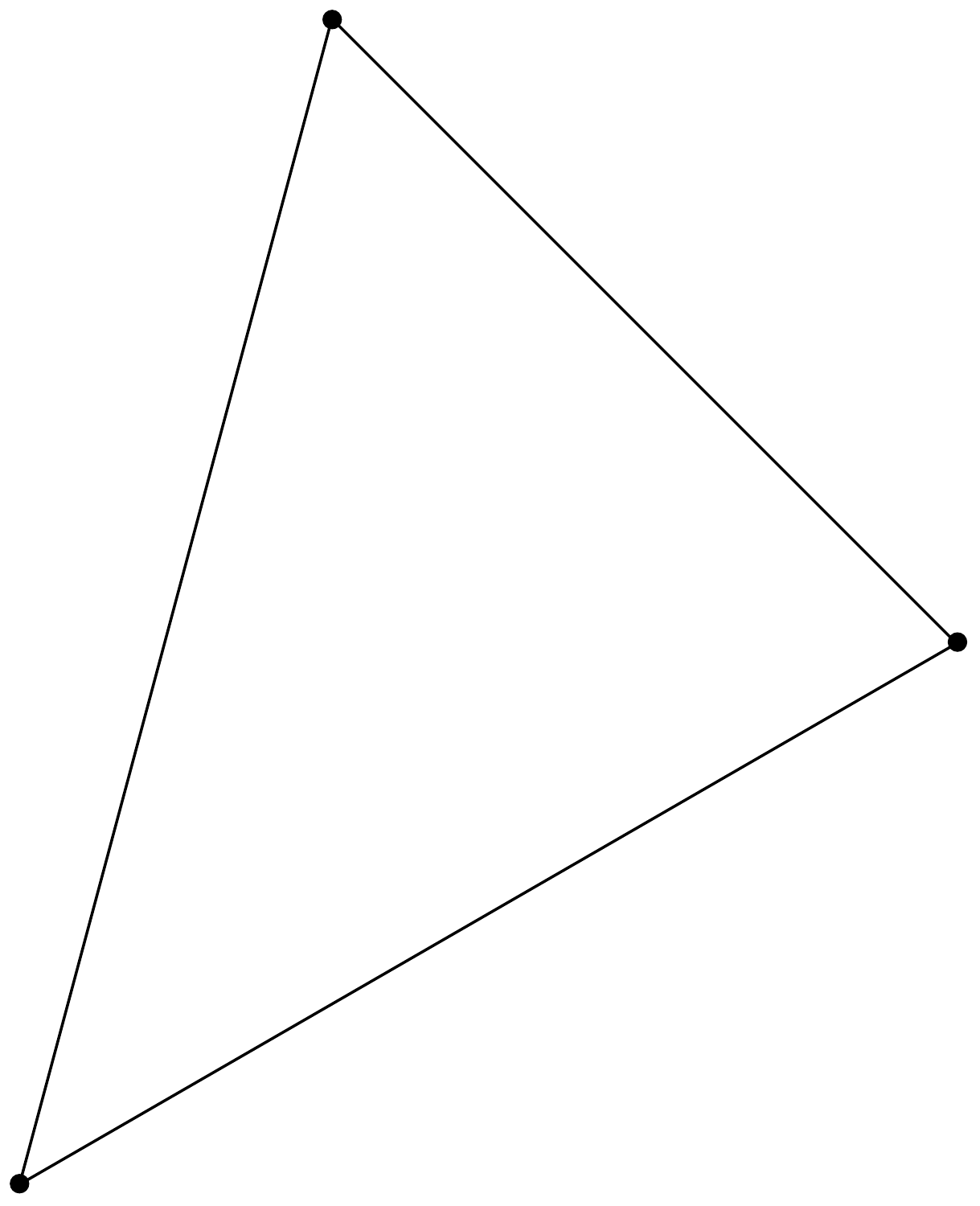}
	\begin{picture}(0,0)
	\put(-4,55){{$p_0$}}
	\put(-75,125){{$p_1$}}
	\put(-114,0){{$p_2$}}
	\put(-30,85){$1+\frac{m(\eta-1)}{l}$}
	\put(-55,20){$1+\frac{m(\eta-1)}{n-l}$}
	\put(-95,60){$\eta$}
	\end{picture}
	\caption{The restriction of $\mcD$ to $V_0'$.}\label{fig9}
	\end{figure}
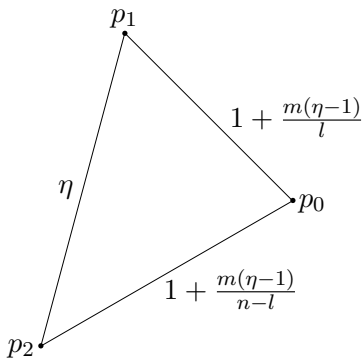
\end{example}

\begin{example}
For the Julia set in Example \ref{example22}, {we have experimentally that} $\eta^{-1}\approx 0.64735$. On the other hand, one can check easily that $T_\mcJ\mathcal{D}_\mcJ=\frac{1}{2}\mathcal{D}_\mcJ$ for the relation $\mcJ$ shown in Example \ref{example318}. In particular, this shows $\eta T_\mcJ\mathcal{D}_\mcJ<\mcD_\mcJ$ as Proposition \ref{prop326} states.
\end{example}

\section{Other finitely ramified Julia sets of rational maps}
In this last section of the paper, we look at some other Julia sets also associated to rational maps $R_\lambda(z)=z^n+\frac{\lambda}{z^m}$ with $n\geq 2$, $m\geq 1$, which are not MS maps. Instead of providing a {full} story as {done in} Section 3, this section {is more explorative}, and we hope that the observations may lead to further studies.

In particular, we focus on {the} simple class of rational maps $R_\lambda$ whose critical set possesses a real fixed point $c$. We have $c=(\frac{n}{m+n})^{\frac{1}{n-1}}$ in this case since $\lambda=\frac{nc^{n+m}}{m}$. Clearly, $c$ is a superattracting fixed point, so the immediate attracting basin of $c$  is excluded from the Julia set $K_\lambda$. See Figure \ref{fig10} for some examples {of these Julia sets}.

\begin{figure}[htp]
	\includegraphics[width=4.5cm]{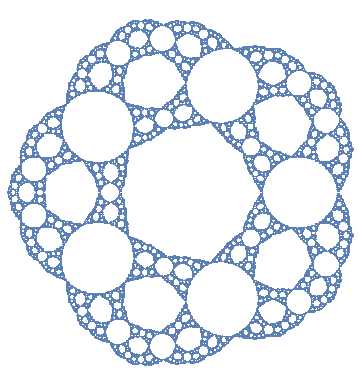}\quad
	\includegraphics[width=4.5cm]{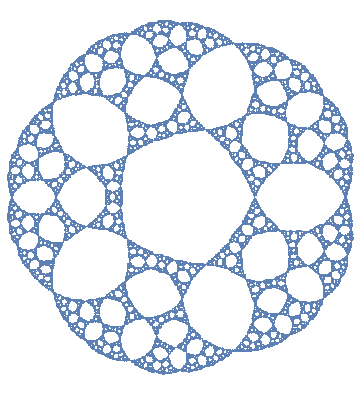}\quad
	\includegraphics[width=4.5cm]{8_1_StationarySet_n_3_m_3.png}
	\begin{picture}(0,0)
	\put(-328,65){$q_0$}
	\put(-310,63){$\tilde{B}_\lambda$}
	\put(-319.2,64.2){$\bullet$}
	\put(-291,64.2){$\bullet$}
	\put(-285,65){$p_0$}
	\put(-370,-10){$m=2, n=3$}
	\put(-236,-10){$m=3, n=2$}
	\put(-99,-10){$m=3, n=3$}
	\end{picture}
	\caption{Julia sets with {a fixed real critical point.}}\label{fig10}
\end{figure} 

We {let} $B_\lambda$ {denote} the immediate attracting basin of $\infty$, $\Gamma_\lambda=R_\lambda^{-1}(B_\lambda)$ {denote} the trap door and $\tilde{B}_\lambda$ {denote} the immediate basin of the real $c$. Then we have two local cut points $p_0,q_0$, namely 
\[\{p_0\}=\overline{B}_\lambda\cap \overline{\tilde{B}_{\lambda}},\quad \{q_0\}=\overline{\Gamma}_\lambda\cap \overline{\tilde{B}_{\lambda}}.\]
Since the Julia set $K_\lambda$ admits the rotation symmetry, we define 
\[p_l=e^{\frac{2l\pi i}{m+n}}p_0,\text{ }q_l=e^{\frac{2l\pi  i}{m+n}}p_0, \text{ for }0\leq l\leq m+n-1.\]
The vertex set $\{p_l,q_l\}_{l=0}^{m+n-1}$ cuts the Julia set $K_\lambda$ into $m+n$ connected components. We denote {by} $K_{\lambda,l}$ the closure of one of the components such that $K_{\lambda,l}$ contains $\{p_l,q_l,p_{l-1},q_{l-1}\}$ for $l=1,2,\dots m+n$, using the cyclic notation $m+n=0$, and call them the \textit{$1$-cells} of $K_\lambda$. 

It is not hard to see that $R_\lambda^{-1}(\{p_l,q_l\}_{l=0}^{m+n-1})$ is a set of $2(m+n)^2$ vertices that contains $\{p_l,q_l\}_{l=0}^{m+n-1}$ and cuts $K_\lambda$ into $(m+n)^2$ pieces. For each pair of $1\leq k,l\leq m+n$, we denote ${\Psi_{k,l}}$ for the local inverse of $R_\lambda$ such that ${\Psi_{k,l}}:{K_{\lambda,k}\to K_{\lambda,l}}$. Then we have 
\[K_{\lambda,l}=\bigcup_{k=1}^{m+n}\Psi_{k,l}(K_{\lambda,k}).\]
We still have the same conjugacy of $R_\lambda$ on the boundary $\beta_\lambda$ of $B_\lambda$ to the angle mapping $\Phi_n$ on the unit {circle} $\mathbb T$ as described in Section \ref{sec2}, and we have $R_\lambda(p_l)=R_\lambda(q_l)$. These facts determine the position of each level-2 cells $\Psi_{k,l}(K_{\lambda,k})$. See Figure \ref{fig13} for an illustration (where we take  $m=2,n=3$).

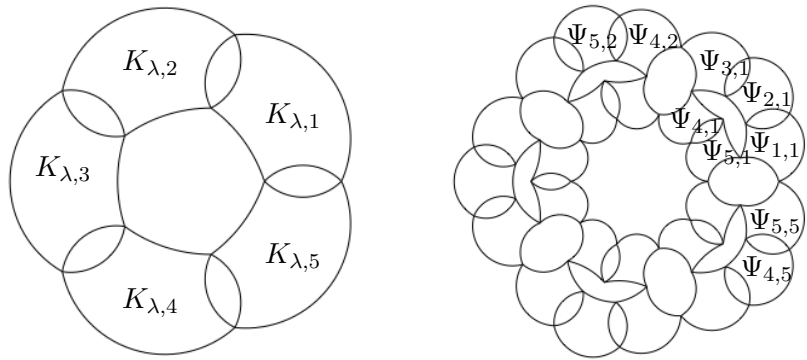
\begin{figure}
	\includegraphics[width=4.8cm]{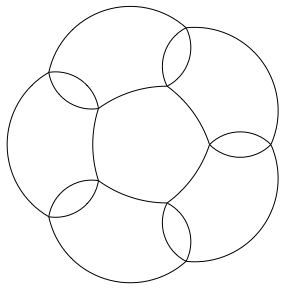}\qquad\quad
	\includegraphics[width=4.8cm]{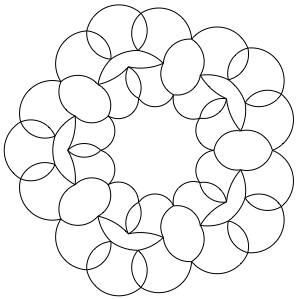}
	\begin{picture}(0,0)
	\put(-210,90){$K_{\lambda,1}$}
	\put(-264,110){$K_{\lambda,2}$}
	\put(-297,70){$K_{\lambda,3}$}
	\put(-264,20){$K_{\lambda,4}$}
	\put(-210,37){$K_{\lambda,5}$}
	
	\put(-27,79){$\Psi_{1,1}$}
	\put(-30,98){$\Psi_{2,1}$}
	\put(-46,110){$\Psi_{3,1}$}
	\put(-57,88){$\Psi_{4,1}$}
	\put(-45,76){$\Psi_{5,1}$}
	
	\put(-73,120){$\Psi_{4,2}$}
	\put(-96,120){$\Psi_{5,2}$}
	
	\put(-27,50){$\Psi_{5,5}$}
	\put(-30,32){$\Psi_{4,5}$}
	\end{picture}
	\caption{An illustration of level-$1$ cells and the $\Psi_{k,l}$ mappings {for $(m,n)=(2,3)$.}}\label{fig13}
\end{figure} 

 By iterating the mappings in {the} proper {way}, we see that  the diameters of the higher level cells {shrink} to $0$. More precisely, we have 
\[\diam(\Psi_{l_{k-1},l_k}\cdots\Psi_{l_1,l_2} \Psi_{l_{0},l_{1}}(K_{\lambda,l_0}))\to 0,\text{ as }k\to\infty,\] 
for any infinite sequence $\{l_k\}_{k\geq 0}$. This provides us {with} a graph-directed structure of $K_\lambda$ (see \cite{CQ,HN}). 

We consider the forms $\mcE_k,k=1,2,\cdots,m+n$ on $K_{\lambda,k}$'s satisfying the graph-directed invariance, i.e. 
\begin{equation}\label{eqn41}
\mcE_k(f)=\eta\sum_{l=1}^{m+n}\mcE_l(f\circ \Psi_{l,k}),
\end{equation}
for some positive constant $\eta$ independent of $k$.

For simplicity, we still consider the rotationally symmetric solutions, i.e. we require 
\begin{equation}\label{eqn42}
\mcE_k(f)=\mcE_{k+l}(f(e^{-\frac{2l\pi i}{m+n}}\bullet))
\end{equation} 
holds for any pair of $1\leq k,l\leq m+n$. Then (\ref{eqn41}) is simplified to an equation of the form (\ref{eqn32}) (with different contractive mappings of course),
\begin{equation}\label{eqn43}
\mcE_1(f)=\eta\sum_{l=1}^{m+n}\mcE_1(f\circ \Psi_{l,1}(e^{\frac{2(l-1)\pi i}{m+n}}\bullet)).
\end{equation} 
In particular, the existence of a solution to (\ref{eqn41}) is equivalent to the existence of a rotationally symmetric solution to (\ref{eqn41}). This can be easily proven with the Hilbert's projective metric \cite{M2} and the Brouwer fiexed point theorem. See \cite{B},  Proposition 6.21, for a similar result on p.c.f. self-similar sets, whose proof can be easily modified for our purpose.

We again apply Sabot's criteria, Theorem \ref{thm38}, to study the existence of forms that satisfy (\ref{eqn43}). In particular, there are only two non-trivial preserved relations $\mcJ_1$ and $\mcJ_2$ on $V:=\{p_0,p_1,q_0,q_1\}$, depicted in Figure \ref{fig11}. 

\vspace{0.15cm}
1. $\mcJ_1$ consists of two equivalence classes $\{p_0,q_0\}$ and $\{p_1,q_1\}$.

2. $\mcJ_2$ consists of two equivalence classes $\{p_0,p_1\}$ and $\{q_0,q_1\}$. 
\vspace{0.15cm}

\begin{figure}[htp]
	\includegraphics[width=3cm]{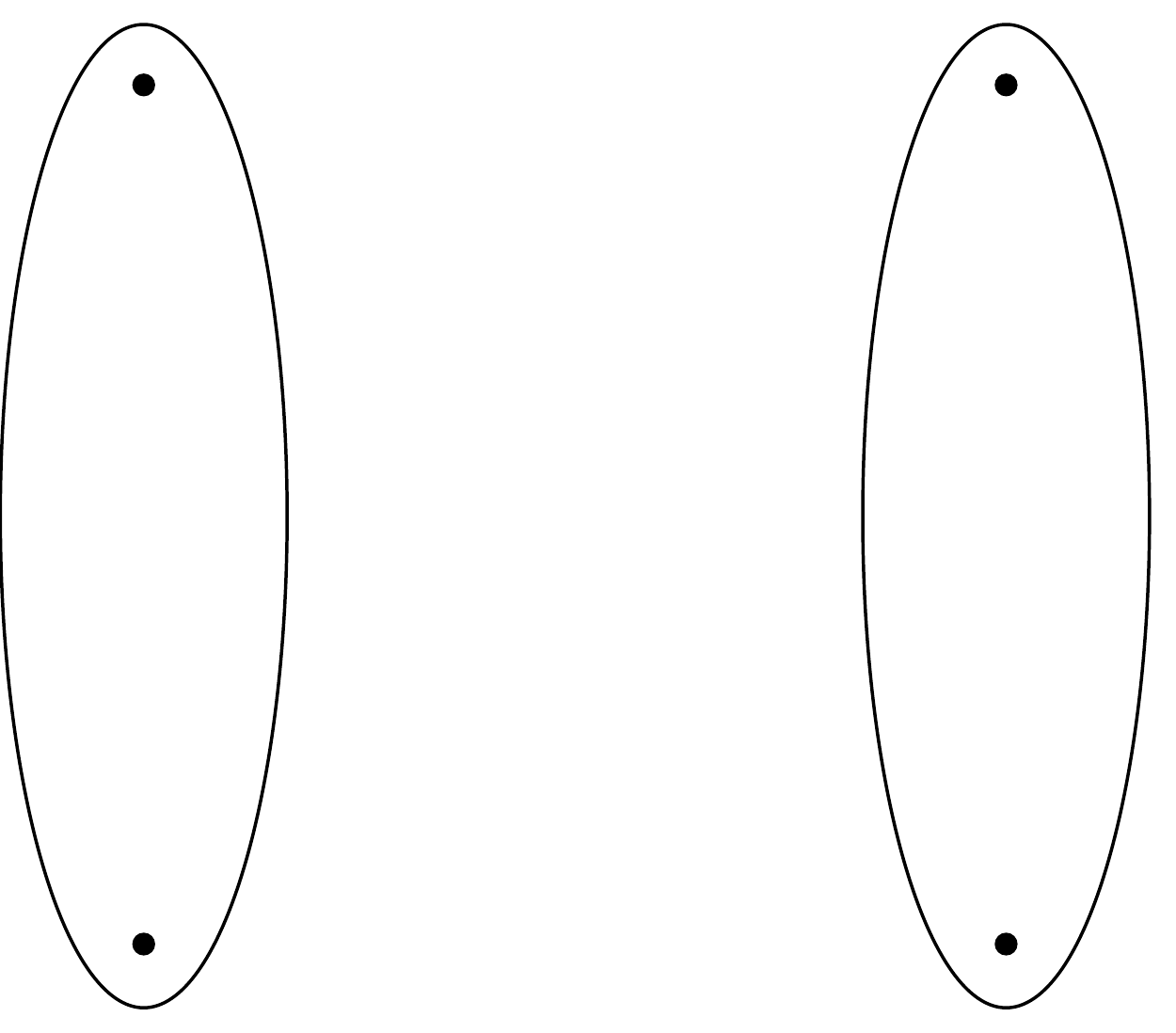}\qquad\qquad\qquad
	\includegraphics[width=3cm,angle=90]{J1.pdf}
	\begin{picture}(0,0)
	\put(-196,-10){$\mcJ_1$}
	\put(-49,-10){$\mcJ_2$}
	\end{picture}
	\caption{The non-trivial preserved relations on $\{p_0,p_1,q_0,q_1\}$.}\label{fig11}
\end{figure}

It is easy to compute the exact values of the $\overline{\rho},\underline{\rho}$'s for $\mcJ_1$ and $\mcJ_2$. 

\begin{proposition}\label{prop42}
For $\mcJ_1,\mcJ_2$ defined above, we have 
	\[\underline{\rho}_{\mcJ_1}=\overline{\rho}_{\mcJ_1}=\frac{1}{2},\quad \underline{\rho}_{V/\mcJ_1}=\overline{\rho}_{V/\mcJ_1}=\frac{1}{m}+\frac{1}{n}, \quad\overline{\rho}_{\mcJ_2}=\frac{1}{n},\quad\underline{\rho}_{V/\mcJ_2}=\frac{mn}{m+n}.\]

As a consequence, we have

(a). There {are} no  graph-directed invariant forms on $K_\lambda$ if $\frac{1}{m}+\frac{1}{n}<\frac{1}{2}$. In addition, the same result holds for $m=n=4$.

(b). There is a unique rotationally symmetric graph-directed invariant forms on $K_\lambda$ if $\frac{1}{m}+\frac{1}{n}>\frac{1}{2}$. 
\end{proposition}

\noindent\textbf{Remark.}  The above proposition follows directly from Sabot's Theorem once we have calculated the exact values of $\overline{\rho},\underline{\rho}$'s for $\mcJ_1$ and $\mcJ_2$. The only unclear case is the critical case $\frac 1m+\frac 1n=\frac 12$. Indeed, there are $3$ possible choices: $(m,n)=(3,6), (6,3)$ or $(4,4)$. The $m=n=4$ case can be studied directly by computation, which is tricky and long. We claim the non-existence result in this case without providing the details. For the $(m,n)=(3,6)$ and $(6,3)$ cases, experiments indicate that there is no solution to (\ref{eqn43}).

\begin{example}
We can compute the {unique} rotationally {invariant} symmetric form on $K_\lambda$ when $m=1$. See Figure \ref{fig12} for some typical such Julia sets. In particular, the Julia set corresponding to the $m=1,n=2$ case is homeomorphic to the double cover of the Sierpinski gasket, and as we shall see has the same renormalization constant.

\begin{figure}[htp]
	\includegraphics[width=4.5cm]{8_1_StationarySet_n_2_m_1.png}\qquad\qquad
	\includegraphics[width=4.5cm]{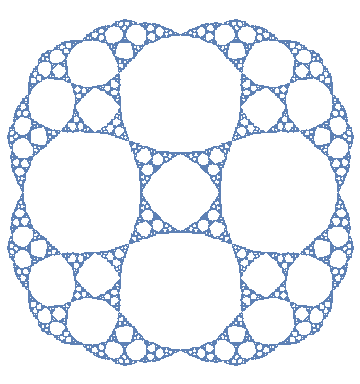}
	\begin{picture}(0,0)
	\put(-266,-10){$m=1, n=2$}
	\put(-99,-10){$m=1, n=3$}
	\end{picture}
	\caption{Some Julia sets with $m=1$.}\label{fig12}
\end{figure}
	
{Computing the} exact solution is tedious in general, but the renormalization constant $\eta$ is surprisingly  concise:
\[\eta=\frac{2n+1}{n+1}.\]
\end{example}

As we can see, for a rational map $R_\lambda$ possessing a fixed critical point, its associated Julia set is quite different from {those of MS maps}. We  leave the more general case, for example $R_\lambda$ possessing a periodic critical point, for future studies.

\section*{Acknowledgments}
{It is the wish of the authors} to thank Prof. Fei Yang for {the} helpful comments {related to} complex dynamics.

\bibliographystyle{amsplain}

\end{document}